 \documentclass[draft]{article}

\usepackage{amsmath,amsfonts,amsthm,amssymb,amscd,cancel,color}
\usepackage{enumitem}
\usepackage{verbatim}

\renewcommand{\Re}{\mathop{\rm Re}\nolimits}
\renewcommand{\Im}{\mathop{\rm Im}\nolimits}

\theoremstyle{plain} \newtheorem{theorem}{Theorem}[section]
\newtheorem{lemma}[theorem]{Lemma}
\newtheorem{proposition}[theorem]{Proposition}
 \theoremstyle{definition}
\newtheorem{definition}[theorem]{Definition} \theoremstyle{remark}
\newtheorem{remark}[theorem]{Remark} 
 
\newcommand{\R}{{\mathbb R}}

\newcommand{\N}{{\mathbb N}}

\newcommand{\resto}{{\mathcal R}} \def\im{{\rm i}}

\newcommand{\C}{\mathbb{C}}

\def\({\left(}
\def\){\right)}
\def\<{\left\langle}
\def\>{\right\rangle}

\numberwithin{equation}{section}

\begin{document}

  \title{ On   small energy stabilization     in the NLKG with a trapping potential }

 \author {Scipio Cuccagna, Masaya Maeda   and Tuoc V. Phan}

 \maketitle
\begin{abstract}   We consider  a nonlinear Klein Gordon equation (NLKG) with short range  potential
   with eigenvalues and  show that in the contest of complex valued solutions the
   small standing waves  are attractors for small solutions of the NLKG. This extends
     the results already known for the nonlinear Schr\"odinger equation and for the nonlinear Dirac equation.  In addition, this extends a result
   of Bambusi and Cuccagna (which in turn was an extension of a result by Soffer  and Weinstein) which considered only real valued solutions of the NLKG.  \end{abstract}

\section{Introduction}\label{sec:intr}
We consider    the following   nonlinear Klein Gordon equation  initial value problem in $ (t,x) \in \R \times \R^3$:
\begin{equation}\label{eq:NKGE}
  \left\{\begin{matrix}
  \dot v  +H u +m^2 u +  |u|^2   u =0 \text{  and } \dot u=v ,   \\
u(0 )=u_0 \in H^1(\R ^3, \C)   \text{  and }   v(0 )=v_0  \in L^2(\R ^3, \C)  .
\end{matrix}\right.
\end{equation}
where $H=-\Delta + V$  with $V\in \mathcal{S}(\R ^3, \R )$, the space of real valued and Schwartz functions.

We recall that   if
$H$  has eigenvalues, then the solutions of the linear version of \eqref{eq:NKGE}, that is with $|u|^2   u$ omitted,
are  formed by various  uncoupled oscillators   and by scattering continuous modes.
In particular, each
eigenvalue can be associated to an invariant linear space in  $ H^1 \times L^2$.  Such invariant spaces
exist  near the origin also for the nonlinear problem   \eqref{eq:NKGE} in the form of topological disks
which are tangent to the linear spaces at the origin.  In the case of the nonlinear Schr\"odinger  equation,
\cite{CM1}, and of the nonlinear Dirac  equation,
\cite{CT}, it as been shown that the union of these disks is an attractor  for all  small energy solutions.
We will prove the same result for \eqref{eq:NKGE}.  This, as in \cite{CM1,CT}, is non trivial  because
one might  have  expect \eqref{eq:NKGE} to have complicated quasi periodic solutions as for  
discrete equations, e.g. \cite{maeda}.

\noindent A partial version of the problem addressed here  has been considered in \cite{SW3,bambusicuccagna} (for a further addition to the result in  \cite{SW3} see also \cite{AS}). In particular, in \cite{bambusicuccagna}  it is shown in quite some generality that when $u_0$ and $v_0$ are real valued then any small energy solution $(u(t), v(t))$ scatters to the origin, that is $ \rho =0$ in Theorem \ref{thm:small en} below. Here   we generalize
this because we consider $u_0$ and $v_0$ complex valued. In this latter  case   $(u(t), v(t))$ can scatter  to
small standing waves, which are always    not   real valued  and therefore could not be seen  in \cite{SW3,bambusicuccagna}.
Since the attracting set is more complicated than just a single point,   the result we obtain here is substantially different than \cite{bambusicuccagna}, and the proof is more elaborate.  Obviously, our paper is motivated by our interest on discrete and
continuous modes interactions. Indeed, if $H$ has no eigenvalues, then the scattering to 0 of all small energy solutions
is a standard consequence of the results on wave operators in Yajima \cite{Y1}.

\noindent Before further comments we introduce some notation, the hypotheses and our main results.

\noindent For $ g,h:\R^3\to \C^i$ ($i=1,2$ with $g=(g_1,g_2)$ and $h=(h_1,h_2)$ if $i=2$), we use the real inner product
 \begin{equation}\label{eq:bilf}
    \langle   g |h\rangle = \begin{cases}\Re \int _{\R ^3} g(x) \overline{h(x)} dx,\quad for\  i=1,\\
\Re \int_{\R^3} \(g_1(x) \overline{h_1(x)}+g_2(x) \overline{h_2(x)}\)\,dx,\quad for \ i=2.\end{cases}
 \end{equation}
 We introduce the Japanese bracket $\langle x \rangle := \sqrt{1+|x| ^2} $ and the spaces
defined by the following norms:
\begin{align} & L^{p ,s}(\R ^3, X ) \text{  defined with  }
  \| u \| _{L^{p,s}(\R ^3, X ) }  :=  \|   \langle x \rangle ^s   u \| _{L^p (\R ^3, X)} ;   \nonumber \\
 &  H^k(\R ^3, X  ) \text{  defined with  }
  \| u \| _{H^k(\R ^3, X  ) }  :=  \|   \langle x \rangle ^k   \widehat{{u}} \| _{L^2 (\R ^3, X   )},
 \text{ where  $ \widehat{{u}}$   the   Fourier transform of $u$;}
   \nonumber\\
& \Sigma ^k(\R ^3, X ) :=L^{2 ,k}(\R ^3, X )\cap   H^k(\R ^3, X )  \text{   with  }  \| u \|^2 _{\Sigma ^k} = \| u \| _{H^{k }(\R ^3, X  ) } ^2+ \| u \|^2 _{L^{2,k}(\R ^3, X   ) },\label{eq:Sigma}
 \end{align}
where $X=\C, \C^2$ or $\R$.
We assume the following.
\begin{itemize}
  \item[(H1)]
$V\in\mathcal{S}(\R^3,\R )$, where $\mathcal S (\R^3,\R ) = \cap _{m\ge 0}\Sigma ^m(\R ^3, \R )$ is the space of  Schwartz functions.

\item[(H2)]
$\sigma_p(H) $ is formed by points $-m^2< e_1<e_2< e_3 \cdots < e_n<0 $, where $\sigma_p(H)$ is the set of point spectrum of $H$.   We also assume that all the eigenvalues have multiplicity 1. Moreover,
 0 is neither an eigenvalue nor a resonance (that is, if $Hu=0$  with $u\in  C^\infty$ and     $|u(x)|\le C|x|^{-1} $ for a fixed $C$, then $u=0$).

\item[(H3)] Set $\omega _{j} = \sqrt{m^2+e_j}$
and consider the smallest $N\in \N$ s.t. $N \min \{ \omega _j \pm \omega _l : j\geq l \} \ge 2m $.
We assume that $|\sum_{j=1}^n m_n \omega_n|\neq m 	$ for all  $|\mathbf m|\leq 4N+6$.
\item[(H4)]  The    Fermi Golden Rule (FGR)  holds:   the expression \begin{equation*}       \sum _{L\in \Lambda}  L \<
 {G}_{L }  | \delta(-\sigma_3 B-L)  \sigma _3{G}_{L}
   \>
 ,
\end{equation*}
which is defined    in the course of the paper (for  $\Lambda \subset \R _+$   see under \eqref{equation:FGR3},  for  $ {G} _L$ see \eqref{eq:pos2})
and   is shown in Sect. \ref{discrete} to be  non--negative and to  equal the l.h.s. of \eqref{eq:FGR}, is assumed here to
satisfy   inequality \eqref{eq:FGR}.

\end{itemize}

We introduce constants \begin{equation}  \label{eq:veceig}  \begin{aligned} & \widetilde{\omega}=(\widetilde{ \omega} _{1},...,\widetilde{ \omega} _{2n})  \text{ with } \widetilde{ \omega} _{J} = (-1) ^{\kappa _J} \omega _{J}, \text{ and } \\& \kappa _J=\left\{\begin{matrix}
0 \text{ for $J=1,...,n$} ,\\
1\text{ for $J=n+1,...,2n$},
\end{matrix}\right.  \text{ we set }   { \omega} _{J} =\left\{\begin{matrix}
 \omega _{J}\text{ for $J\le n$} ,\\
 \omega _{J-n}\text{ for $J>n$}.
\end{matrix}\right. \end{aligned}
\end{equation}

To each  $e_j$   we associate an eigenfunction $\phi _j$. We choose them s.t.  $\langle  \phi _j|{\phi }_k\rangle  =\delta _{jk} $.
 Since we can,  we also   choose the $\phi _j$  to be all real valued.
To each  $\phi _j$ we associate   nonlinear bound states. For a proof
of the following standard result see Appendix A in \cite{CM1}.
Notice that the real analyticity with respect to $z$ is also immediate from the fact that the nonlinearity is analytic, see also \cite{maeda}.
Moreover, for
 the spaces $\Sigma^t  $ and the notation $D_X (v, \delta)$, we refer to Sect. \ref{subsec:notation}.

\begin{proposition}[Bound states]\label{prop:bddst} There exists
$a _0>0$ such that for every $j\in \{ 1,\cdots,n\}$, and $\forall z~\in~D_\C ( 0, a _0 )$, there is a unique  $Q_{jz } \in \mathcal{S}(\R^3, \C ) := \cap _{t\ge 0}\Sigma^t (\R^3, \C )$, s.t.
\begin{equation}\label{eq:sp}
\begin{aligned}
&H Q_{jz } + |Q_{jz }|^2  Q_{jz }= E_{jz }Q_{jz } \quad , \quad
Q_{jz }=  z \phi_j + q_{j z}, \ \langle q_{j z}|{\phi}_j\rangle =0,
\end{aligned}
\end{equation}
and s.t. we have for any $r\in \N$:
\begin{itemize}
\item[\textup{(1)}]    $(q_{jz },E_{jz }) \in C^\omega ( D_\C ( 0, a _0 ), \Sigma ^r(\R^3,\C)\times \R )$;   and $q_{jz } =  z \widehat{q}_{j  } (|z|^2)$, with
$\widehat{q}_{j  } (t^2 ) =t ^2\widetilde{q}_{j }(t^2)$,  for $\widetilde{q}_{j } (t ) \in C^\omega (   ( - {a _0 }^{2}, {a _0 }^{2}), \Sigma ^r (\R ^3, \R ) )$,  and  $E_{jz }  =E_{j } (|z|^2)$ with $E_{j } (t ) \in C^\omega (   ( - {a _0 }^{2}, {a _0 }^{2}),   \R  )$;
\item[\textup{(2)}]    $\exists$ $C >0$ s.t.
$\|q_{jz }\|_{\Sigma ^r} \leq C |z|^3$, $|E_{jz }-e_j|<C | z|^2$.
\end{itemize}

\end{proposition}

 \noindent We set
 \begin{equation}\label{eq:othJ}
  \omega _{Jz }  := \left\{\begin{matrix}
\sqrt{ m^2+ E_{Jz }}  \text{ for $J\le n$} ,\\
\sqrt{ m^2+ E_{(J-n)z }}  \text{ for $J> n$}
\end{matrix}\right.    \, , \, \phi _{J  }  := \left\{\begin{matrix}
\phi _{J  }   \text{ for $J\le n$} ,\\
\phi _{J -n }   \text{ for $J> n$ } \end{matrix}\right.\, , \, q _{J z }  := \left\{\begin{matrix}
q _{J z }  \text{ for $J\le n$} ,\\
q _{(J-n)z }  \text{ for $J> n$ ,}
\end{matrix}\right.
 \end{equation}
and further set $\tilde \omega_{J}$ as \eqref{eq:veceig}.
 Let $Q_{Jz }=  z \phi_J + q_{J z}$ for all $J\le 2n$.
 Then $u(t,x) = e ^{\pm \im t \omega _{Jz }} Q_{Jz } (x)$ are solutions to
 \eqref{eq:NKGE1}. For  $z\in D_{\C^{2n }}(0,b _0) $   for $b _0>0$    sufficiently small we consider
\begin{align}\label{eq:standw}
\Phi_J [z ]:=(Q_{J z},  \im \widetilde{\omega} _{J z} Q_{J z})\ \mathrm{for}\ z\in\C\ \mathrm{and}\     \Phi [z ]:= \sum_{J=1}^{2n} \Phi_J[z_J] \ \mathrm{for}\ z\in\C^{2n}.
\end{align}
Notice that by the definition of $\Phi$ and Proposition \ref{prop:bddst}, we have the gauge property $\Phi[e^{\im \theta}z]=e^{\im \theta}\Phi[z]$.

In  $L^2(\R^3,\C )\times L^2(\R^3,\C)$  we consider    the symplectic form
\begin{equation}\label{eq:Omega}
\Omega(U|V ):=2\<\mathbb J^{-1}U|V\>,\
\mathrm{where}\ \mathbb J =\begin{pmatrix} 0 & -1 \\ 1 & 0\end{pmatrix}.
\end{equation}

\begin{definition}\label{def:contsp}
Let       $z=(z_1,..., z_{2n})$, $z_{JR}=\Re z_J$,  $z_{JI}=\Im z_J$.  For any $z\in D_{\C^{2n }}(0,b _0) $  with $b _0 >0$ be sufficiently small
and for $\partial _{ {JA}}=\partial _{z_{ JA}}$
we set
\begin{equation}\label{eq:contsp}
\begin{aligned}
\mathcal{H}_c[z ] :=  \{( \eta _1, \eta _2) \in L^2\times L^2 :
\Omega(\partial _{{JA}} \Phi [z ] |( \eta _1, \eta _2) )=0   \text{ $\forall$ $J $ and $A=R,I$} \}.
\end{aligned}
\end{equation}
\end{definition}
 \begin{remark}
Notice that     $\mathcal{H}_c[0 ] = L_c^2\times  L_c^2$ where
  $L^2_c= P_cL^2$ is the continuous space associated to $H$ in   $L^2(\R^3,\C)$.
Here,
\begin{align}\label{eq:contproj}
P_cu=u-\sum_{j=1}^n\(\<u,\phi_j\>\phi_j+\<u,\im \phi_j\>\im\phi_j\).
\end{align}

\end{remark}

\begin{remark}
By $\Phi[e^{\im \theta}z]=e^{\im \theta}\Phi[z]$ we have $(\eta_1,\eta_2)\in \mathcal H_c[z]\Leftrightarrow (e^{\im \theta}\eta_1,e^{\im \theta}\eta_2)\in \mathcal H_c[e^{\im\theta}z]$.
\end{remark}

More generally, given a space $\Sigma ^{k} (\R^3,\C)$  we write
  $\Sigma ^{k}_c=  P_c\Sigma ^{k}(\R^3,\C)$.
Notice that $P_c$ can be defined on $\Sigma^k$ even if $k$ is negative.

Throughout the paper, we say that a  pair $(p,q)$   is
{\it admissible} when \begin{equation}\label{admissiblepair}
2/p+3/q= 3/2\,
 , \quad 6\ge q\ge 2\, , \quad
p\ge 2. \end{equation}
The following theorem is our main result.
\begin{theorem}\label{thm:small en}    Assume $(\mathrm{H1})$--$(\mathrm{H4})$. Then there exist  $\epsilon _0 >0$ and $C>0$ such that  if we set  $\epsilon :=\| (u (0),v(0))\| _{H^1\times L^2}<\epsilon _0  $, then the  solution  $(u(t),v(t))$ of  \eqref{eq:NKGE} can be written uniquely for all times   and with $( \eta (t) , \xi (t)) \in
\mathcal{H}_c[z(t)]$ as
\begin{equation}\label{eq:small en1}
\begin{aligned}&   (u(t),v(t))= \Phi [z (t)]+  ({\eta} (t) ,{\xi} (t)) .
\end{aligned}
\end{equation}
 such that
there exist  a unique $J_0$, a
$\rho  _+\in [0,\infty ) ^{2n}$ with $\rho_{+J}=0$ for $J\neq J_0$,
s.t.
  $| \rho  _+ | \le C  \epsilon  $, and there exists $(u_+
,v_+ )$ with $\| (u_+ ,v_+ ) \| _{H^1\times L^2}\le C \epsilon
$
 and such that,  for $u_{lin}(u_+,v_+)(t) = K'_0(t)
u_+ + K_0(t) v_+$ where
 $K_0(t)=\frac{\sin (t \sqrt{-\Delta +m^2})}{\sqrt{-\Delta +m^2}} $,  we have
\begin{equation}\label{eq:small en3}
\begin{aligned}&     \lim_{t\to +\infty}\|  ( {\eta} (t ) , {\xi } (t )) -(u_{lin}(u_+,v_+)(t) , \frac{d}{dt}u_{lin}(u_+,v_+)(t)) \|_{H^1\times L^2 }=0      , \\&
 \lim_{t\to +\infty} |z_J(t)|  =\rho_{+J}  .
\end{aligned}
\end{equation}
Furthermore, we have $({\eta}  , {\xi} )= (\widetilde{\eta}  , \widetilde{\xi} ) + \mathbf{{A}}(t,x) $  s.t.
    for    all admissible pairs  $(p,q)$
 \begin{equation}\label{eq:small en2}
\begin{aligned}&     \| z \| _{L^\infty _t( \mathbb{R}_+ )}+ \|
(\widetilde{\eta}  , \widetilde{\xi} )   \| _{L^p_t( \mathbb{R}_+,W^{ \frac{{1}}{q}  - \frac{{1}}{p} ,q}_x  \times W^{ \frac{{1}}{q}  - \frac{{1}}{p} -1,q}_x)} \le C \epsilon   \ , \\& \| \dot z _J +\im  \omega _{ J }z_J \|  _{L^\infty _t( \mathbb{R}_+ ) } \le C  \epsilon ^2,\
\end{aligned}
\end{equation}
and s.t. $\mathbf{A}(t,\cdot )\in \Sigma ^2 \times \Sigma ^2$  for all $t\ge 0$, and
\begin{equation}\label{eq:small en4}
\begin{aligned}&      \lim_{t\to +\infty}\| \mathbf{A}(t,\cdot )   \|_{\Sigma ^2\times \Sigma ^2 }=0 .
\end{aligned}
\end{equation}

\end{theorem}
As we have already mentioned above, this theorem extends \cite{SW3}, which dealt with $(u_0,v_0)$ both real valued,   $n=1$ in
(H2), and  with the further restriction that $3\omega _1>m$. It extends  also   \cite{bambusicuccagna}  in the sense that while \cite{bambusicuccagna}  allows even more
general spectra than (H2),  \cite{bambusicuccagna}  considers only $(u_0,v_0)$ both real valued. The fact that  $(u_0,v_0)$ both are real valued. and the convergence is to 0  allows
    in      \cite{bambusicuccagna} a simpler    choice    of coordinates  systems. This is key  because, like in \cite{bambusicuccagna,CM1,CT}, the main point in the proof of
Theorem   \ref{thm:small en}     consists in finding an appropriate system of coordinates where it is easier
  to extract an appropriate effective hamiltonian.
  Indeed, the purpose is to exploit  the hamiltonian
nature of the equation  to prove damping through dispersion of most of the $z_J$'s.

\noindent A simple example of this damping mechanism is given is the following, which we quote from \cite{CM2}. Consider the hamiltonian
\begin{equation}   \label{model1}    \begin{aligned} &
  \mathcal{H} (z,h )=
	   |z |^2+  \|    \nabla h \| ^{2}_{L^2}  +   |z|^2   \overline{{z}}    \int _{\R^2}
  G   (x) h  (x) dx  +   |z|^2   z
  \int _{\R^2}\overline{G}   (x)\overline{h } (x) dx        \end{aligned}
\end{equation}
 for which the equilibrium $(0,0)$ is asymptotically stable, as we   sketch  heuristically now.
We have
   \begin{align}    &
  \im \dot h = - \Delta   h + |z|^2   z  \overline{G} \text{  and}
 \label{intstr-11}  \\&
\im \dot z
=  z   +   2 |z|^2
\int _{\R^2} h (x)
  {G}(x)  dx   +             z^2
\int _{\R^2} \overline{h}(x)
  \overline{{G}}(x)  dx   . \label{intstr11}
\end{align}
If we  set
\begin{equation*}\label{intstr-1}
   h = - |z|^2   zR ^{+}_{-\Delta}(1)\overline{G} +g  \text{   for  }R ^{+}_{-\Delta}(1) = \lim  _{\varepsilon \to 0^+} R  _{-\Delta}(1+\im \varepsilon )
\end{equation*}
      and we substitute in \eqref{intstr11},   ignoring   $g$  since it is smaller,  we get
\begin{equation*}    \begin{aligned} &
\im \dot z
=  z   -  2  |z|^4 z
\int _{\R^2}
  {G}  R ^{+}_{-\Delta}(1)\overline{G}  dx  -     |z|^4 z
\int _{\R^2}
  \overline{{G}}   R ^{-}_{-\Delta}(1) {G} dx   .
\end{aligned}  \end{equation*}
Recall $R ^{\pm}_{-\Delta}(\lambda ) = P.V. (-\Delta- \lambda) ^{-1} \pm \im \pi \delta (-\Delta- \lambda)$  for any $\lambda >0$.
Multiplying by $\overline{z}$ and taking imaginary part we get
\begin{equation} \label{intstr12}  \begin{aligned} &
 \frac{d}{dt} |z|^2= - 2\pi \mathfrak{c}  |z|^8 \text{  with }
\mathfrak{c} =\int _{\R^2}
  {G}   \delta (-\Delta  -1) \overline{G}  dx \ge 0   .
\end{aligned}  \end{equation}
We conclude that $\mathfrak{c}\ge 0$ by the following formula, see ch.2 \cite{friedlander}:   \begin{equation} \label{eq:termc}  \begin{aligned} &   \mathfrak{c} =   \frac{1}{2 } \int _{|\xi |= 1} |\widehat{G } (\xi
)|^2 d\sigma (\xi ).
\end{aligned}\end{equation}
 Assuming   $\mathfrak{c} >0$ (the meaning of hypothesis (H4) is basically this), which is generically true,
  then   \eqref{intstr12} yields the explicit formula
\begin{equation*}
    |z(t)| ^2= \frac{|z(0)| ^2}{(1+6 \pi \mathfrak{c} |z(0)| ^2t) ^{\frac{1}{3}}} .
\end{equation*}
In the meantime, $h$ scatters because one can apply Strichartz estimates to  \eqref{intstr-11}.

 So the work in all the papers  \cite{bambusicuccagna,CM1,CT}, as well as here, consists
 in finding a coordinate system where \eqref{eq:NKGE}
 has hamiltonian that, up to negligible  error terms, is  similar to \eqref{model1}. In \cite{bambusicuccagna} this is simpler
 because there the attractor set is formed just by the vacuum solution of \eqref{eq:NKGE}. Here
 as well as in  \cite{CM1,CT} the attractor is more complex and leads to   solutions with complicated trajectories
 around the attractor, as shown in
  \cite{TY3}, which treats for the NLS  a  similar problem but only for  $n=2$ in  (H2).

  As discussed
  in  \cite{CM1},  the approach in  \cite{TY3} involving    guesses on  the trajectory of a solution, the turns of the solution away  from unstable
  standing waves and, usually,  its final convergence either to 0 or to a stable standing wave, appears
  a considerably difficult task under our hypothesis (H2), which is
  much more complex combinatorially than the situation in \cite{TY3}. In this paper we frame the problem as in  \cite{CM1}
  and extend to the NLKG equation the NLS result obtained in  \cite{CM1}, exactly as in  \cite{CT}  the result of
   \cite{CM1}  has been extended to   Dirac equations. It turns out that the NLKG presents no significant
   new problems with respect to the NLS. In this sense this    paper  contains applications of ideas introduced already in  \cite{CM1}.
   However, due to the importance of the NLKG, we think   the result in the present paper
   significant nonetheless. This in view of the fact that, apart from
   \cite{SW3,bambusicuccagna} and \cite{AS},  not much has been written about the asymptotic stability of
   standing waves of the   NLKG. This in contrast to the rather extensive literature on asymptotic stability of standing waves  of the NLS, e.g.   \cite{SW1,SW2,PiW}, \cite{BP1,BP2}, \cite{TY1}--\cite{TY4}, \cite{SW4}, \cite{zhousigal,GW1,GW2}, \cite{CM,Cu2}  and therein,
   and various other papers on the problem treated in  \cite{CM1}  about the NLS, e.g. \cite{GNT,NPT,N}. There are even some papers on asymptotic stability for the nonlinear Dirac equation, e.g. \cite{PelinovskyStefanov,boussaidcuccagna,CT}.

Now we give a quick description of the proof  of Theorem \ref{thm:small en}. Following ideas from \cite{GNT}, subsequently elaborated in
\cite{CM1}, we find a natural system of coordinates  $(z_1, ..., z_{2n},\Xi   ) $  for \eqref{eq:NKGE}, which comprise both discrete modes $z_J $ for $J=1,...2n$  and   continuous modes   $ \Xi $. Since one of the discrete modes possibly does not decay, early in the paper a new auxiliary variable
$\mathbf{Z}$ is introduced.  In the proof it is shown that all the components of $\mathbf{Z}$ decay to 0.
 The components of $\mathbf{Z}$ are the products $z_J \overline{z}_K$  with $J\neq K$. The role of the discrete mode $z$   in
\eqref{model1}  is taken by $\mathbf{Z}$ in  the hamiltonian of \eqref{eq:NKGE}. Some elementary but essential lemmas about monomials
in the variable $\mathbf{Z}$ are then introduced in Sect. \ref{subsec:comb}.

\noindent In the simplified setup of \cite{bambusicuccagna,SW3} the initial  coordinates are Darboux for the symplectic structure
in the problem, but not here. Like in \cite{Cu2,CM1} here  we need instead to  change coordinates to reduce to Darboux coordinates.  Like
in \cite{CM1}  all the coordinate changes in this paper satisfy \begin{equation} \label{eq:rel1}\begin{aligned} &   z'_1 = z_1 + O( z \Xi ) + O(\Xi ^2) + \sum _{J\neq K}O( z_J z_K )  , ...,  z'_{2n} = z_{2n} + O( z \Xi ) + O(\Xi ^2)  + \sum _{J\neq K}O( z_J z_K ), \\& \Xi ' =\Xi  + O( z \Xi ) + O(\Xi ^2) + \sum _{J\neq K}O( z_J z_K ) .  \end{aligned}
\end{equation}
 We take an appropriate expansion of the energy in terms of these coordinates emphasizing  $(z,\mathbf{Z},\Xi)$. Like in \cite{CM1}, in the Darboux coordinates  an important cancelation
occurs in the energy, and specifically in the 2nd line of \eqref{eq:enexp10} the summations start  from  $l=1$
and not from $l=0$ like in the analogous expansion in \eqref{eq:enexp1}. We provide a simplified explanation with respect to
\cite{CM1} for this cancelation
in   the course of
Lemma \ref{lem:KExp1}. The hamiltonian  is now    of the same type  of that in \cite{CM1} so applying Theorem 5.9 \cite{CM1}
we obtain a hamiltonian somewhat similar to \eqref{model1}. There is a mechanism of nonlinear discrete continuous interaction
similar to that sketched for \eqref{model1} that yields the stabilization mechanism of Theorem
\ref{thm:small en}, with $\mathbf{Z}(t) \stackrel{t \to \infty}{\rightarrow}0  $  and scattering of $\Xi$.
Thanks to the structure of the coordinate changes \eqref{eq:rel1} this behavior  transfers  from one coordinate system to the other and yields the
decay of all the $z_J(t)$ except for at most one  and the scattering of the continuous components.

\noindent For the scattering of the continuous modes we use a number of Strichatz and smoothing estimates stated in \cite{bambusicuccagna}, in part proved there and in part
gathered from the literature,  mainly from \cite{danconafanelli}, but see  \cite{bambusicuccagna} for further references.

\noindent Finally, one thing that we do not accomplish here and which was proved instead in \cite{CM1}  is   the instability of the excited
states.

\section{Notation, coordinates and resonant sets} \label{section:set up}

\subsection{Notation}
\label{subsec:notation}
\begin{itemize}
\item We denote by $\N =\{ 1,2,...\}$  the set of natural numbers  and  set $\N_0= \N\cup \{ 0\}  $.
\item
We denote $z=( z_1,\dots ,z _{2n})$, $|z|:=\sqrt{\sum_{j=1} ^{2n}|z_j|^2}$.

\item Given a Banach space $X$, $v\in X$ and $\delta>0$ we set
$
D_X(v,\delta):=\{ x\in X\ |\ \|v-x\|_X<\delta\}.
$

\item
Let $A$ be an operator on $L^2(\R^3)$.
Then $\sigma_p(A)\subset \C$ is the set of eignvalues of $A$ and $\sigma_e(A)\subset \C$ is the   essential spectrum of $A$.

\item
For $m<0$ and integer   we set $\Sigma ^m=(\Sigma  ^{-m})'$.  Notice  that the  spaces   $\Sigma ^r $
can be equivalently defined using   for $r\in \R$    the norm
 $
      \| u \| _{\Sigma ^r}  :=  \|  ( 1-\Delta +|x|^2)   ^{\frac{r}{2}} u \| _{L^2}        .   $

\item
For $f:\C^{2n }\to \C$ we set
$\partial_{JA} f(z):=\frac{\partial}{\partial z_{J A}}f(z)$,  $\partial _{ J }:=\partial _{z_J}$  and  $\partial _{\overline{J} }:=\partial _{\overline{z}_J}$.
Here as customary   $\partial _{z_J}  = \frac 12 (\partial _{JR}-\im  \partial _{JI} )$ and    $\partial _{\overline{z}_J}  = \frac 12 (\partial _{JR}+\im  \partial _{JI} )$.

\item Occasionally we use a single  index $\ell =J, \overline{J}$. To define $\overline{\ell}$  we use the convention $\overline{\overline{J}}=J$.
We will also write   $z_{\overline{J}}=\overline{z}_J$.

\item  We will consider vectors $z=(z_1, ..., z _{2n})\in \C^{2n}$ and  for  vectors $\mu , \nu \in (\N \cup \{ 0 \} ) ^{2n}$ we set $ z^\mu \overline{z}^\nu  := z_1^{\mu _1} ...z_{2n}^{\mu _{2n}}\overline{z}_1^{\nu _1} ...\overline{z}_{2n}^{\nu _{2n}}$.  We will set $|\mu |=\sum _j \mu _j$.

\item   We have  $dz_J  =dz_{JR}+\im dz_{JI}$,    $d\overline{z}_J  =dz_{JR}-\im dz_{JI}$.

\item  Given two Banach spaces $X$ and $Y$ we denote by $\mathcal{L}(X,Y)$ the   space
of bounded linear operators $X\to Y$ with the norm of the uniform operator topology.

\item We set $\mathbf L^2:=L^2\times L^2$, $\mathbf \Sigma^r:=\Sigma^r\times \Sigma^r$ and $\mathbf H^s:=H^s\times H^s$.

\item
$\sigma_3=\begin{pmatrix} 1 & 0 \\ 0 & -1 \end{pmatrix}$.

\end{itemize}

\subsection{Coordinates}
\label{subsec:coord}
In the following, we fix $r_0>0$ sufficiently large.

\noindent
A preliminary step for the choice of coordinates is the following
standard ansatz.

\begin{lemma}\label{lem:decomposition}
There exist $c _0 >0$, and $C>0$ s.t.  for all $(u,v) \in H^1\times L^2$   with $\| (u,v) \|_{H^1\times L^2}<c  _0 $, there exists a unique pair $(z ,\widetilde \Xi )\in   \C^{2n }\times  ( (H^1\times L^2) \cap \mathcal{H}_{c}[z])$
s.t.
\begin{equation}\label{eq:decomposition1}
\begin{aligned} &
(u,v)=\Phi[z] +\widetilde \Xi   \\&\text{ with }  |z | +\| \widetilde \Xi \|_{H^1\times L^2}\le C \| (u,v)\|_{H^1\times L^2} .
\end{aligned}
\end{equation}
The map  $u \to (z ,\widetilde \Xi )$  is $C^\omega (D _{H^1}(0, c _0 ),
 \C ^{2n} \times   H^1\times L^2  )$,  and  satisfies  the   gauge property: \begin{equation}\label{eq:decomposition3}
\begin{aligned} &
 \text{$z(e^{\im \vartheta} u,e^{\im \vartheta} v )=e^{\im \vartheta} z( u,v)
$,  and    $\Psi (e^{\im \vartheta} u,e^{\im \vartheta} v )=e^{\im \vartheta} \Psi ( u,v)$ } .
\end{aligned}
\end{equation}

\end{lemma}
\proof
We consider  for $J=1,..., 2n$ and  $A=R,I$ the functions
\begin{equation*}
\begin{aligned}&
  F _{JA} ( u,v, z ):=\Omega\big (\partial _{JA} \Phi [z ] | ( u,v)- \Phi [z ] \big  )  .
\end{aligned}
\end{equation*}
The $ F _{JA}$ is analytic in
$\mathbf L^2 \times D_{\C  ^{2n}} (0,  b _0)$  for   the $b _0$
in Def. \ref{def:contsp}.
 We have \begin{equation}\label{eq:ion31}
\begin{aligned}&
 F _{JR} (0, 0, z) = -4  \widetilde{\omega} _{J  }  z_{JR}+O(z^3)  , \, F _{JI} (0, 0, z) = - 4\widetilde{\omega }_{J 0}  z_{JI}+O(z^3).
\end{aligned}
\end{equation}
    The  map $(u,v)\to z$ in   $C^\omega ( D _{\mathbf L^2}(0, c _0)  ,
 \C  ^{2n}   )$ for a $c _0>0$ sufficiently small  is obtained by implicit function theorem.  All the other statements
 are equally elementary. For a proof in a similar set up see \cite{CM1}.
   \qed

We introduce
\begin{equation}\label{eq:eqopB}
 B:= \sqrt{-\Delta +V+m^2} \, , \quad  {B}^{-\frac{\sigma _3}{2}} =    \begin{pmatrix} {B}^{-\frac{1}{2}} &
0 \\
0 & B^{\frac{1}{2}} .
 \end{pmatrix}
\end{equation}

We need a system of independent coordinates, which
 the $(z,\Psi  )$
in   \eqref{eq:decomposition1}  are not. The following lemma is used to complete (later in Lemma \ref{lem:systcoo})  the $z$ with a continuous coordinate.

\begin{lemma} \label{lem:contcoo}
There exists
$d _0>0$ such that there exist $C_{JA}\in C^\omega(D_{\C^{2n}}(0,d_0),\Sigma^{r_0})$ for $J=1,\cdots,2n$ and $A=R,I$ such that
\begin{itemize}
\item[\textup{(i)}]
For $R[z ]$ defined by
\begin{align}\label{eq:contcoo21}
R[z]\Xi = {B}^{-\frac{\sigma _3}{2}}\Xi + \sum _{J'=1}^{2n}\sum _{A'=R,I} \< C _{J'A'}[z]|\Xi \>  \partial_{J'A'}\Phi [0],
\end{align}
we have $R[z ]: \mathbf H ^{\frac{1}{2}}_c\to (H ^{1}\times L^2)\cap \mathcal{H}_c[z ]$ and
\begin{align}\label{eq:contcoo22}
 \left.  {{B}}^{ \frac{\sigma _3}{2}}\circ (P_c\times P_c)\right|_{\mathcal{H}_c[z ]}=R[z ]^{-1}.
\end{align}
\item[\textup{(ii)}]
$ \| C _{JA}[z]  \|_{\Sigma ^r}  \le c_r | z|^2$ for all $r\geq 0$.
\item[\textup{(iii)}]
$R[e^{\im \theta}z]=e^{\im \theta}R[z]e^{-\im \theta}$ for all $\theta \in\R$.
\end{itemize}
\end{lemma}

\begin{proof}
We define $C_{JA}'[z]\Xi \in \R$ ($J=1,\cdots,2n, A=R,I$) to be the unique solution of the system
\begin{align}\label{eq:contcoo5}
\Omega\(\partial_{JA}\Phi[z]\left |{B}^{-\frac{\sigma _3}{2}}\Xi + \sum _{J'=1}^{2n}\sum _{A'=R,I}  (C _{J'A'}'[z]\Xi     )   \partial_{J'A'}\Phi [0]\right . \)=0,\ J=1,\cdots,2n,A=R,I,
\end{align}
where  $\partial _{J R }\Phi [0] = ( \phi _J, \im \widetilde{\omega} _J\phi _J)$  and $\partial _{J I }\Phi [0] = \im ( \phi _J, \im \widetilde{\omega} _J\phi _J)=\im\partial _{J R }\Phi [0] $.
Rewriting \eqref{eq:contcoo5}, we have
\begin{align}\label{eq:contcoo6}
\sum _{J'=1}^{2n}\sum _{A'=R,I} \Omega(\partial_{JA}\Phi[z]|\partial_{J'A'}\Phi[0])C_{J'A'}'[z]\Xi=-\Omega(\partial_{JA}\Phi[z]|B^{-\frac{\sigma_3}{2}}\Xi)
\end{align}
Since the coefficient matrix is invertible for $z$ sufficiently small, we can solve \eqref{eq:contcoo5}.
Furthermore, from \eqref{eq:contcoo6}, we see that $C'_{J'A'}[z]\Xi$ can be written as $\<C_{JA}[z]|\Xi\>$ for some $C_{JA}[z]\in \Sigma^r$ for arbitrary $r$.
The real analyticity with respect to $z$ follows from the real analyticity of $\Phi[z]$.
Notice that since $|\Omega(\partial_{J'A'}\Phi[z]|B^{-\frac{\sigma_3}{2}}\Xi)|\lesssim |z|^2$, we have (ii).

\noindent Next we define  $R[z]$  by   \eqref{eq:contcoo21}  using these $C_{JA}[z] $.
 From \eqref{eq:contcoo5} we obtain $R[z ]: \mathbf H ^{\frac{1}{2}}_c\to (H ^{1}\times L^2)\cap \mathcal{H}_c[z ]$.
Furthermore, \eqref{eq:contcoo22}   follows from the form of $R[z]$ and the uniqueness of the solution of \eqref{eq:contcoo5}.

We will postpone (iii) to the appendix of this paper.

\end{proof}

By $R[z]$ given in Lemma \ref{lem:contcoo}, we have a system of independent coordinates which we needed.

\begin{lemma} \label{lem:systcoo}  For the $d  _0>0$ of Lemma \ref{lem:contcoo}
the map $\mathcal F:(z ,\eta,\xi  )\mapsto U:=(u,v)$  defined  for  $\Xi =  ( \eta  ,  \xi )$
by
\begin{equation} \label{eq:systcoo1}\begin{aligned}
&
\mathcal F(z,\Xi)=\Phi [z]+  R[z ]    \Xi     \text{ for $(z, \Xi )\in D_{\C ^{2n }}(0 , d _0) \times   \mathbf H ^{\frac{1}{2}}_c  $}\end{aligned}
\end{equation}
is  with values in   $H^1 \times  L^2 $  and  is   $C^\omega$.
Furthermore, there is a $d_1>0$ such that    for $(z, \Xi )\in B_{\C^{2n }}(0 , d _1) \times  B_{\mathbf H ^{\frac{1}{2}}_c }(0 , d _1) $
the above map is a  diffeomorphism and
\begin{equation} \label{eq:coo11}
  |z| +\| \Xi\| _{\mathbf H ^{\frac{1}{2}}_c} \sim   \| \mathcal F(z,\Xi) \| _{ H^1 \times  L^2}.
\end{equation}
Finally, setting $\Psi[z,\Xi]:=\Phi[z]+(R[z]-B^{-\frac{\sigma_3}{2}})\Xi$, we have $\Psi \in C^\omega(D_{\C^{2n}}\times \mathbf \Sigma_c^{-r_0}, \mathbf \Sigma^{r_0})$.

\end{lemma}

\subsection{Some simple combinatorics} \label{subsec:comb}

We now recall from \cite{CM1} the following definition, where we are introducing the auxiliary quantity $\mathbf{Z}$.
\begin{definition}\label{def:comb0} Given $z\in \C^n$, we  denote by $\mathbf{Z}$ the vector (or the matrix)    with entries $(  z_{J}  \overline{{z}}_K)$ with $J,K\in [1,{2n}]$ but only with pairs of  indexes  with $J\neq K$.  Here  $\mathbf{Z}\in   L$, where $L$ is the subspace of $ \C ^{n_0} =\{  (a_{J,K} ) _{J, K =1, ..., n} : J\neq  K   \}   $ for $n_0={2n} ({2n}-1),$ with $(a_{J,K} ) \in  L $ iff $a_{J,K} = \overline{a}_{K,J}$ for all $J,K$.
 For  a multi index $\textbf{m}=\{  {m}_{JK}\in \N _0 :J\neq K  \}$ we set $\textbf{Z}^{\textbf{m}}=\prod  (z_{J}  \overline{{z}}_K) ^{ {m}_{JK}} $  and $ |\mathbf{m}| :=\sum _{J,K}m  _{JK} $.
\end{definition}

We will  think formally of  $\mathbf{Z}$ as an auxiliary variable and we will
 be dealing with polynomials in the variable $\mathbf{Z}$.
Some of the monomials in $\mathbf{Z}$ will be reminder terms. In order to   distinguish
between the monomials which are reminder terms and the ones which aren't, we  need the following definition.

\begin{definition}\label{def:setM}   Consider the set of multi indexes  $\textbf{m}$
as in Definition  \ref{def:comb0}. Consider for any $K\in \{ 1, ..., 2n  \}$ the set
\begin{equation} \label{eq:setM0} \begin{aligned} &  \mathcal{M}_K (r)=
\{   \textbf{m} :   \left | \sum _{L=1} ^{2n} \sum _{J=1} ^{2n} m  _{LJ} ( \widetilde{\omega} _L - \widetilde{\omega} _J) -  \widetilde{\omega} _K \right |>m   \text{ and  }  |\mathbf{m}|   \le r \},  \\& \mathcal{M}_0 (r)=
\{   \textbf{m} :     \sum _{L=1} ^{2n} \sum _{J=1} ^{2n} m  _{LJ} ( \widetilde{\omega} _L - \widetilde{\omega} _J) =0   \text{ and  }  |\mathbf{m}|   \le r \}  .
\end{aligned}
\end{equation}
Set now
\begin{equation} \label{eq:setM1} \begin{aligned} &   {M}_K(r)=
\{    (\mu , \nu )\in \N _0 ^{2n}  \times  \N _0 ^{2n} :   \exists \textbf{m}   \in \mathcal{M}_K (r) \text{ such that }   z^{\mu} \overline{z}^{\nu} = \overline{z}_K   \mathbf{Z} ^{\textbf{m} }  \}  , \\& M(r)= \cup  _{K=1} ^{2n} {M}_K(r).
\end{aligned}
\end{equation}
We also set $M= M (2N+4)$, and
\begin{equation} \label{eq:setMhat} \begin{aligned} &   M_{min}=
\{    (\mu , \nu )\in M :      (\alpha  , \beta  )\in M  \text{ with } \alpha _J\le \mu _J \text{ and } \beta _J\le \nu _J \ \forall \ J \Rightarrow  (\alpha  , \beta  ) = (\mu , \nu )  \}  .
\end{aligned}
\end{equation}
\end{definition}
It is easy to see that if $(\mu , \nu )\in M_{min}$, then for any $J$ we have $\mu  _{J} \nu _{J}=0$. Indeed,  first of all    $(\mu , \nu )\in M(r)$    if and only if $|\nu |= |\mu |+1$, $|\mu | \le r$  and $|(\mu - \nu )\cdot  \widetilde{\omega} |>m$. Now,  if  $\mu  _{J_0}\ge 1$ and $ \nu _{J_0} \ge 1$  then, by subtracting from both of them a unit and leaving unchanged the other
coordinates, we obtain another pair $(\alpha , \beta )\in M_{min}$ with  $ \alpha _J\le \mu _J $ and $ \beta _J\le \nu _J$ for all $J$ but with $(\alpha  , \beta  ) \neq (\mu , \nu )$, so that $(\mu , \nu )\not \in M_{min}$.

\begin{lemma} \label{lem:comb1} We have the  following facts.

\begin{itemize}
\item[\textup{(1)}]
Let   $\mathbf{a}=( a  _{LJ}) \in \N _0 ^{n_0}  $ s.t.  for the $N$ in (H3)
\begin{equation}\label{eq:comb-3}
    \sum _{ L<J\le n}a  _{LJ} +\sum _{ L\le n <J }a  _{ JL} +\sum _{n<L<J}a  _{ JL} >N.
\end{equation}
  Then   for any $K$,  we have
(notice the switch  in the indexes)
\begin{equation}\label{eq:comb3}
\begin{aligned}
&  \sum _{ L<J\le n}a  _{LJ} (\omega _L -\omega _J)-\sum _{ L\le n <J}a  _{ JL} (\omega _L +\omega _J) +\sum _{ n<L<J}a  _{ JL} (\omega _L -\omega _J) +   \omega _K< -m   .
\end{aligned}
\end{equation}

\item[\textup{(2)}] Consider $\mathbf{m}\in \N _0 ^{n_0} $  with  $|\mathbf{m}|   \ge 2N+3$  and the monomial $ z _{J_0} \mathbf{Z}^{\mathbf{m}}$.
    Then $\exists$  $\mathbf{a}, \mathbf{b}\in \N _0 ^{n_0} $ s.t. for both $\mathbf{f}=\mathbf{a},\mathbf{b}$
\begin{equation}\label{eq:comb4}
\begin{aligned}
&  \sum _{ L<J\le n}f  _{LJ} +\sum _{ L\le n <J }f  _{ JL} +\sum _{n<L<J}f  _{ JL}     =N+1 \text{ } ,\\&
  f  _{LJ}  =0  \text{ for all  $(L,J)$ not  in the above sum},  \text{ and }
   a  _{LJ}   +  b  _{LJ} \le m  _{LJ} + m  _{ JL}   \text{ for all } (L,J).
\end{aligned}
\end{equation}
Moreover, there are  two indexes $ (K,S)$ s.t.
 \begin{equation}\label{eq:comb5}
\begin{aligned}
& \sum _{L,J }  a  _{LJ} (\widetilde{\omega} _L -\widetilde{\omega} _J) - \widetilde{\omega} _K<- m   \text{ and }    \sum _{L,J }  b  _{LJ} (\widetilde{\omega} _L -\widetilde{\omega} _J) - \widetilde{\omega} _S<- m,
\end{aligned}
\end{equation}
 and for $|z|\le 1$
  \begin{equation}\label{eq:comb6}
\begin{aligned}
& |z  _{J_0} \mathbf{Z}^{\mathbf{m}}|\le |z_{J_0}| \  |z _K \mathbf{Z}^{\mathbf{a}}|  \  |z _S \mathbf{Z}^{\mathbf{b}}| .
\end{aligned}
\end{equation}

\item[\textup{(3)}] For $\mathbf{m}$ with $|\mathbf{m}|   \ge 2N+3$,  there exist $(K,S)$, $\mathbf{a} \in \mathcal{M}_K$ and $\mathbf{b} \in \mathcal{M}_S$ s.t.  \eqref{eq:comb6}  holds.

    \end{itemize}
\end{lemma}
\proof The inequality \eqref{eq:comb3} follows immediately   by the definition of $N$ and
\begin{equation*}
\begin{aligned}
& \sum _{L<J}  a  _{LJ} (\omega _L -\omega _J) + \omega_K< -\min \{ \omega _J -\omega _L : J>L \} N +m \le -m.
\end{aligned}
\end{equation*}
We now consider the monomial  $ z  _{J0}\mathbf{Z}^{\mathbf{m}}$.   Since  $|\mathbf{m}|\ge 2N+3$, there are vectors
$\mathbf{c}, \mathbf{d}\in \N _0 ^{2n_0} $ such that  $|\mathbf{c}|=|\mathbf{d}| = N+1$ with $c  _{LJ}   +  d  _{LJ} \le m  _{LJ}  $   for all  $(L,J)$.

\noindent Focusing on $ {\mathbf{c}}$ we define  $\mathbf{a}$ as follows:
\begin{equation*}
\begin{aligned}
&\text{  if $L<   J\le n$  let $a _{LJ}=c_{LJ}+c_{JL}$ and $a _{JL}=0$;}\\&
   \text{  if $L\le n <J$  let $a _{LJ}=0$ and $a _{JL}=c_{LJ}+c_{JL}$;}\\&
\text{  if $  n <L<J$  let $a _{LJ}=0$ and $a _{JL}=c_{LJ}+c_{JL}$}.
\end{aligned}
\end{equation*}
Notice that
\begin{equation*}
\begin{aligned}
&      \widetilde{\omega} _L -\widetilde{\omega} _J = \left\{\begin{matrix}   {\omega} _L - {\omega} _J \text{  if $L<   J\le n$ }\\
  \omega _L+{\omega} _J \text{  if $L\le n <J$ }\\  -\omega _L+{\omega} _J \text{  if $  n <L<J$ } \end{matrix}\right.
\end{aligned}
\end{equation*}
and that
\begin{equation*}
\begin{aligned}
&     N+1 = \sum _{L,J}  a  _{L,J}    =\sum _{ L<J\le n}a  _{LJ} +\sum _{ L\le n <J }a  _{ JL} +\sum _{n<L<J}a  _{ JL}.
\end{aligned}
\end{equation*}
Hence, by claim (1) \begin{equation*}
\begin{aligned}
&    \sum _{L,J }  a  _{LJ} (\widetilde{\omega} _L -\widetilde{\omega} _J) - \widetilde{\omega} _K \\&
=  \sum _{L <   J\le n } a  _{LJ} \overbrace{(\widetilde{\omega} _L -\widetilde{\omega} _J)}^{\omega _L-{\omega} _J} +\sum _{L\le n <J }  a_{ JL}\overbrace{(\widetilde{\omega} _J -\widetilde{\omega} _L)} ^{-(\omega _L+{\omega} _J)}+ \sum _{ n <L<J}  a_{ JL}\overbrace{(\widetilde{\omega} _J -\widetilde{\omega} _L)}^{\omega _L-{\omega} _J} - \widetilde{\omega} _K\\&
\le   \sum _{L <   J\le n } a  _{LJ}(\omega _L-{\omega} _J) -\sum _{L\le n <J }  a_{ JL} { (\omega _L+{\omega} _J)}+ \sum _{ n <L<J} a_{ JL} (\omega _L-{\omega} _J) +{\omega} _K<-m.
\end{aligned}
\end{equation*}
We can define $\mathbf{b}$  from $\mathbf{d}$ similarly. Then, we have
 \begin{equation}\label{eq:comb7}
\begin{aligned}
& z  _{J_0}\mathbf{Z}^{\mathbf{m}} = z _{J_0}   z^\mu \overline{z} ^\nu \mathbf{Z}^{\mathbf{c}}  \mathbf{Z}^{\mathbf{d}} \text{ with $|\mu | >0 $, and  $|\nu | >0 $.}
\end{aligned}
\end{equation}
Therefore, for $z_K$ a factor of $z^\mu $ and  $\overline{z}_S$ a factor of $\overline{z}^\nu $,
for $|z|\le 1$ we have from \eqref{eq:comb7}
\begin{equation*}
\begin{aligned}
& |z _{J_0} \mathbf{Z}^{\mathbf{m}}| \le | z _{J_0}  | \ |  z_K   \mathbf{Z}^{\mathbf{c}} | \  | \ z_S \mathbf{Z}^{\mathbf{d}}  | =  | z _{J_0}  | \ |  z_K   \mathbf{Z}^{\mathbf{a}} | \  | \ z_S \mathbf{Z}^{\mathbf{b}}  | .
\end{aligned}
\end{equation*}
Furthermore,  \eqref{eq:comb4} is satisfied. Moreover, since our $(\mathbf{a},\mathbf{b})$ satisfy
$\mathbf{a} \in \mathcal{M}_K$ and $\mathbf{b} \in \mathcal{M}_L$,
claim (3) is a consequence of claim (2).

\qed

\section{The energy functional}
\label{sec:invariants}

 For $U=(u,v)$, the equation \eqref{eq:NKGE}  admits the following energy
\begin{equation}\label{eq:inv}
\begin{aligned}
&
  E(U)= \<\mathcal AU|U\>+E_P(u),
\\&\mathrm{where}\
\mathcal A=\begin{pmatrix} B^2 & 0 \\ 0 & 1\end{pmatrix}
  \text{  and }
E_P(u ):=  2^{-1}  \int_{\R^3}  |u|^4 d x.
\end{aligned}
\end{equation}
   We have   $ {E}\in C^\omega ( H^1( \R ^3, \C  )\times  L^2( \R ^3, \C  ), \R  )$.    We denote by $dE$ the Frech\'et derivative of $E$.
We also define $\nabla E$ by
\begin{align}\label{eq:grad1}
dE(U)X=\<\nabla E(U)|X\>.
\end{align}
Similarly, for $F=F(z,\Xi)$, we set $\nabla_\Xi F$ by $d_\Xi F(z,\Xi)\tilde \Xi =\<\nabla_\Xi F(z,\Xi),\tilde \Xi\>$, where $d_\Xi F$ is the Frech\'et derivative with respect to $\Xi$.

We set $E^0(z,\Xi):=\mathcal F^* E (z,\Xi)$ where $\mathcal F$ is given in Lemma \ref{lem:systcoo}.
\begin{lemma}
  \label{lem:EnExp}
For $(z,\Xi )\in D_{\C^{2n }}(0 , d  _1) \times ( D_{\mathbf H ^{\frac{1}{2}} }(0 , d  _1)\cap \mathcal{H}_c[0]) $
    we  have
\begin{align}
   &E^0 (z,\Xi)=     \sum _{J=1}^{2n}E (\Phi_J[z_J])   +   \langle B \Xi |  {\Xi} \rangle+ E _P( B^{-\frac{1}{2}}\eta)
       \nonumber
       \\&
 +
	 \sum _{l =0} ^{\infty}\sum _{J =0}^{2n}	 \sum _{
	  |\textbf{m}|=l+1 }  \textbf{Z}^{\textbf{m}}   a_{J \textbf{m} }^{(0)}( |z_J |^2 )+
   \sum _{l =0} ^{\infty}\sum _{J=1}^{2n} 	 \sum _{
	  |\textbf{m}|=l  }
        \langle
  {z}_J \textbf{Z}^{\textbf{m}} G_{J\textbf{m}}^{(0)}(|z_J|^2  )| \Xi \rangle   \label{eq:enexp1}\\&
\nonumber
+ \sum _{ i+j  = 2}        \langle  G_{2 ij }^{(0)} ( z,\Xi )| (B^{-\frac{1}{2}}\eta )^{   i} \overline{B^{-\frac{1}{2}}\eta} ^j\rangle
+  \sum _{ i+j  = 3}   \langle G_{3 ij }^{(0)} ( z,\Xi )|   B^{-\frac{1}{2}}\eta ^{   i} \overline{B^{-\frac{1}{2}}\eta} ^j\rangle
\nonumber   +\resto^{(0)},
  \end{align}
    where $a_{J\mathbf m}^{(0)}(|z_0|^2)$ will mean $a_{J\mathbf m}^{(0)}$ (a constant), $  \langle B \Xi |  {\Xi} \rangle:=\langle B \xi |  {\xi} \rangle + \langle   B \eta |   {{ \eta}} \rangle $ and:
 \begin{itemize}

\item[\textup{(1)}] $  ( a_{j\textbf{m} }^{(0)} ,  G_{jk\textbf{m}} ^{(0)}  )  \in C^\omega (D_{\R}(0,d  _0^2), \C\times  \mathbf \Sigma ^{r_{0}}(\mathbb{R}^3, \mathbb{C}   )
) $ with $|a_{J\mathbf m}^{(0)}(|z_{J}|)|\leq c |z_J|^2$ and $G_{J0}^{(0)}(0)=0$;

 \item[\textup{(2)}]$(G_{2\textbf{m} ij }^{(0)},G_{d ij }^{(0)},\resto^{(0)})\in C^\omega (D_{\C^{2n}\times \mathbf \Sigma^{-r_0}}(0,d  _0),
\Sigma ^{r_{0}} (\mathbb{R}^3, \mathbb{C}   )\times
\Sigma ^{r_{0}} (\mathbb{R}^3, \mathbb{C}   )\times
\mathbf \Sigma ^{r_{0}} (\mathbb{R}^3, \mathbb{C}   ))$;

\item[\textup{(3)}]
$G_{2 ij }^{(0)} (0)=0$ and $|\resto^{(0)}|\leq c|z|\|\Xi\|_{\mathbf \Sigma^{-r_0}}^2$.


\end{itemize}
\end{lemma}
\proof
First, by Lemma \ref{lem:systcoo}, we have $\mathcal F(z,\Xi)=\Psi[z,\Xi]+B^{-\frac{\sigma_3}{2}}\Xi$.
By Taylor expansion, we have
\begin{align}\label{eq:expansion1}
E(U)=&E(\Psi[z,\Xi])+\<\nabla E(\Psi[z,\Xi])|{B^{-\frac{\sigma_3}{2}}\Xi}\>\nonumber\\&+ \int_0^1(1-t)\<\nabla^2 E(\Psi[z,\Xi]+t B^{-\frac{\sigma_3}{2}}\Xi)B^{-\frac{\sigma_3}{2}}\Xi \ |\  {B^{-\frac{\sigma_3}{2}}\Xi}\>  dt.
\end{align}
Since $\Psi \in C^\omega(D_{\C^{2n}}\times \mathbf\Sigma_c^{-r }, \mathbf\Sigma^{r })$ for arbitrary $r$, we can expand the 1st term and the 2nd term as
\begin{align}\label{eq:expansion2}
&
	\sum _{l =0} ^{\infty}	 \sum _{J =0}^{2n} \sum _{
	  |\textbf{m}|=l }  \textbf{Z}^{\textbf{m}}    a_{J \textbf{m} }^{(0)}( |z_J |^2 )+
  \sum _{l =0} ^{\infty} \sum _{J=1}^{2n} 	 \sum _{
	  |\textbf{m}|=l  }
      \langle  {z}_J \textbf{Z}^{\textbf{m}}
    G_{J\textbf{m}}^{(0)}(|z_J|^2  )| \Xi \rangle    + \resto^{(0)}  (z, \Xi  ) ,
\end{align}
where $\resto$ satisfies the estimate in (3) and $a_{0\mathbf m }^{(0)}(|z_0|^2)=a_{0\mathbf m}^{(0)}$ (constant).
Notice that the pure term  of $z_J$ (i.e.\ $ a_{J0}$) corresponds to $E(Q_{Jz_J},\im\tilde \omega_{Jz_J}Q_{Jz_J})$ and, since there are no constant terms, $a_{00}^{(0)}=0$.

\noindent We now consider the 3rd term of \eqref{eq:expansion1}.
For $\Psi(z,\Xi)=(\psi_1(z,\Xi),\psi_2(z,\Xi))$ and recalling $\Xi =( \eta , \xi )$,
\begin{align}\label{eq:expansion3}
\nabla^2 E(\Psi[z,\Xi]+tB^{-\frac{\sigma_3}{2}}\Xi)=2\begin{pmatrix} B^2 & 0 \\ 0 & 1\end{pmatrix} + \nabla^2 E_P \left (\psi_1(z,\Xi)+B^{-1/2}\eta  \right ) .
\end{align}
The contribution of the 1st term of the r.h.s. of \eqref{eq:expansion3} in the 3rd term of \eqref{eq:expansion1} is
\begin{align}\label{eq:expansion4}
 \int_0^1(1-t)\<2\begin{pmatrix} B^2 & 0 \\ 0 & 1\end{pmatrix}B^{-\frac{\sigma_3}{2}}\Xi \big | {B^{-\frac{\sigma_3}{2}}\Xi}\>=\<B\Xi|{\Xi}\>.
\end{align}
Recall  that $E_p(u)=\frac 1 2 \int |u|^4\,dx$.  Elementary computations show
that $\nabla  E_p\in C^\omega (H^1, H ^{-1}   )$ and $\nabla ^2  E_p\in C^\omega (H^1, \mathcal{ L}(H^1, H ^{-1}   )   )$ are defined by
\begin{equation*}
\begin{aligned}  & \nabla  E_p(u)=2 |u|^2u    \, , \quad   \nabla  ^2 E_p(u)X= 2|u|^2X + 4u \Re ( u \overline{X}),
\end{aligned}
 \end{equation*}
 so that $\langle \nabla  ^2 E_p(u)X|X\rangle = 4 \langle |u|^2\big | |X|^2\rangle +
 2\langle  u ^2\big |  X ^2\rangle $.
 Then, after elementary computations,  the contribution of the 2nd term of the r.h.s. of \eqref{eq:expansion3} in the 3rd term of \eqref{eq:expansion1} is
\begin{equation} \label{eq:expansion5}
 \begin{aligned}
&   2   \<|\psi_1|^2\big  |\ |B^{-1/2} \eta|^2\>+     \< \psi_1^2\big  |(B^{-1/2} \eta)^2\> +   2 \<\psi_1 \big |   |B^{-1/2} \eta|^2{B^{-1/2} \eta}\>+  E_p(B^{-1/2} \eta).
\end{aligned}\end{equation}
Therefore, combining \eqref{eq:expansion1}--\eqref{eq:expansion5}, we have the conclusion.
\qed

\section{Darboux Theorem}
\label{sec:darboux}

We define $\Omega_0$ by
\begin{align}\label{eq:newsymp}
\Omega_0(X\big |Y)=\sum_{J=1}^{2n}\Omega(d\Phi_J[z_J]X\big |d\Phi_J[z_J]Y)+\Omega(d\Xi\big |d\Xi).
\end{align}

\begin{proposition}[Darboux theorem]\label{prop:darboux}
There exists $d_D>0$ such that there exists $\tilde {\mathfrak F}^D=(\tilde {\mathfrak F}^D_z,\tilde {\mathfrak F}^D_\Xi) \in C^\omega(D_{\C^{2n}\times \mathbf \Sigma^{-r_0}}(0,d_1),C^{2n}\times \mathbf{ \Sigma}^{r_0})
$ such that
\begin{align}\label{eq:propdar}
|\tilde {\mathfrak F}^D_z(z,\Xi)|+\|\tilde {\mathfrak F}^D_\Xi(z,\Xi)\|_{\mathbf \Sigma^{r_0}}\leq c|z|\(|\mathbf Z|+\|\Xi\|_{\mathbf \Sigma^{-r_0}}\),
\end{align}
and $\mathfrak F^D:=\mathrm{Id}+\tilde{ \mathfrak F}^D$ satisfies $\(\mathfrak F^D\)^* \Omega= \Omega_0$.
In addition, we have $\mathfrak F^D(e^{\im \theta}z,e^{\im \theta}\Xi)=e^{\im \theta}\mathfrak F^D(z,\Xi)$.
\end{proposition}

The rest of this chapter is devoted for the proof of Proposition \ref{prop:darboux}.
It consists of three steps.
\begin{enumerate}
\item
Find an appropriate   1-form $\Gamma$ such that $\Omega-\Omega_0=d \Gamma$.
\item
For each $s\in [0,1]$ find the vector field $\mathcal X^s$ s.t. $i_{\mathcal X^s}\Omega_s=-\Gamma$,   where $\Omega_s:=(\Omega+s(\Omega-\Omega_0))$.
\item
Solve $\frac{d}{ds}\mathfrak F^s =\mathcal X^s(\mathfrak F^s)$ with $\mathfrak F^0=\mathrm{Id}$.
\end{enumerate}
Then, $\mathfrak F:=\mathfrak F^1$ will be our desired transform.
This will be seen by
\begin{align}\label{eq:argmoser}
\frac{d}{ds} (\mathfrak F^s)^*\Omega_s =(\mathfrak F^s)^*(\mathcal L_{\mathcal X^s}\Omega_s+\partial_s \Omega_s)=(\mathfrak F^s)^*(di_{\mathcal X^s}\Omega_s+d \Gamma)=0.
\end{align}
The estimate \eqref{eq:propdar} will be deduced from the construction.

\noindent  It is elementary that $d\gamma =\Omega$ and $d \gamma_0=\Omega_0$
for the following 1--forms:
 \begin{equation} \label{eq:1forms}\begin{aligned} &
\gamma (U)X:=2 ^{-1} \Omega  ( U | X ),\quad
\gamma_0(U)X:=2^{-1}\sum_{J=1}^{2n} \Omega(\Phi_J[z_J]|d\Phi_J[z_J]X)+ 2^{-1}\Omega(\Xi|d\Xi).
\end{aligned}\end{equation}
we have .

\begin{lemma}
  \label{lem:1forms}
  For  the $d _0>0$  of Lemma \ref{lem:systcoo}
  set $D=D _{ \C^{2n} \times \mathbf  {\Sigma} ^{-r_0}} (0,d _0) $. Then
  there exist functions $\Gamma_{JA}\in C^\omega(D ,\R)$ and $\Gamma_\Xi\in C^\omega(D ,\mathbf \Sigma^{r_0})$ s.t. $\psi \in C^\omega (D   , \R) $  such that
 \begin{align} &
    \gamma(U)-\gamma_0(U)-d\psi =\sum_{J=1}^{2n}\sum_{A=R,I}\Gamma_{JA}dz_{JA}+\<\Gamma_\Xi|d\Xi\>:=\Gamma.    \label{eq:alpha1}
\end{align}
Further, we have $\Gamma(e^{\im \theta}U)=\Gamma(U)$ for all $\theta\in\R$ and
\begin{align}\label{eq:1forms2}
\sum_{J=1}^{2n}\sum_{A=R,I}|\Gamma_{JA}|+\|\Gamma_\Xi\|_{\mathbf \Sigma^{r_0}}\leq c |z|\(|\mathbf Z|+\|\Xi\|_{\mathbf\Sigma^{-r_0}}\),
\end{align}
\end{lemma}
\proof
Notice that, for  $\tilde R[z]=R[z]-B^{-\frac{\sigma_3}{2}}$, expressing $U$ and $X=dUX$
in terms of \eqref{eq:systcoo1}
we have
\begin{align*}
2 \gamma(U)X=\Omega(U|X)=\Omega(\sum_{J=1}^{2n}\Phi_J[z_J]+B^{-\frac{\sigma_3}{2}}\Xi + \tilde R[z]\Xi\ |\  d(\sum_{K=1}^{2n}\Phi_K[z_K]+B^{-\frac{\sigma_3}{2}}\Xi + \tilde R[z]\Xi)X).
\end{align*}
Therefore,  using  also  $\mathbb J ^{-1}\sigma_3 =- \sigma_3\mathbb J ^{-1}$, which  implies   $B^{-\sigma_3/2}\mathbb J^{-1} B^{-\sigma_3/2}=\mathbb J^{-1}$,
we have
\begin{align}
2(\gamma-\gamma_0)=&\sum_{J=1}^{2n}\Omega(\Phi_J[z_J]\ |\ \sum_{K\neq J}d\Phi_K[z_K]+B^{-\frac{\sigma_3}{2}}d\Xi+d\(\tilde R[z]\Xi\))\nonumber\\&+\Omega(B^{-\frac{\sigma_3}{2}}\Xi\ |\  \sum_{K=1}^{2n}d\Phi_K[z_K]+d\(\tilde R[z]\Xi\))\label{eq:1forms1}\\&
+\Omega(\tilde R[z]\Xi\ |\ d(\sum_{K=1}^{2n}\Phi_K[z_K]+B^{-\frac{\sigma_3}{2}}\Xi + \tilde R[z]\Xi)).\nonumber
\end{align}
Notice   $2(\gamma-\gamma_0)-\sum_{J=1}^{2n}\Omega(\Phi_J[z_J]|B^{\frac{-\sigma_3}{2}}d\Xi)$
 admits an expansion like the   middle part of \eqref{eq:alpha1}, where the coefficients satisfy   \eqref{eq:1forms2}.
 If we write
\begin{align*}
\sum_{J=1}^{2n}\Omega(\Phi_J[z_J]|B^{-\frac{\sigma_3}{2}}d\Xi)=d \sum_{J=1}^{2n}\Omega(\Phi_J[z_J]|B^{-\frac{\sigma_3}{2}}\Xi)+
\Omega(B^{-\frac{\sigma_3}{2}}\Xi|d\Phi_J[z_J])
\end{align*}
for  $\psi= \sum_{J=1}^{2n}\Omega(\Phi_J[z_J]|B^{-\frac{\sigma_3}{2}}\Xi)$  we obtain \eqref{eq:alpha1}--\eqref{eq:1forms2}.
 \qed

\begin{lemma}
  \label{lem:vectorfield0}   Consider the form  $\Omega _s:=\Omega _0+s(\Omega -\Omega _0)$ and set $i_X\Omega _t (Y):=\Omega _t (X|Y)$. Then $(\Omega -\Omega _0)$ and $\Gamma $  extend to forms defined in $    D_{\C^{2n}}(0 , d _0 ) \times \mathbf\Sigma ^{-r_0}_c $   and there is a  $d_1 \in (0, d_0) $ s.t.
  for any   $(s, z, \Xi )\in (-4, 4)\times  D_{\C ^{2n}  \times \mathbf \Sigma ^{-r_0}_c} (0,d_1 )$    there exists
  exactly one solution
    $\mathcal{X}^s  (z, \Xi )\in \mathbf L^2$ of the  equation     $
 i_{\mathcal{X}^s} \Omega _s=-  \Gamma
   $.
     Furthermore,  we have  the following facts.

\begin{itemize}
\item[\textup{(i)}]
     $\mathcal{X}^s  (z, \Xi  )\in \mathbf \Sigma ^{ r_0}$ and if we set $\mathcal{X}^s _{JA} (z, \Xi )=dz_{JA}\mathcal{X}^s  (z, \Xi )$ and  $\mathcal{X}^s _{\Xi} (z, \Xi )=d\Xi \mathcal{X}^s  (z, \Xi)$,   we have $\mathcal X^s_{JA}(z,\Xi)\in C^\omega(D_{\C^{2n}\times \mathbf \Sigma^{-r_0}}(0,d_1)\times(-4,4),\R)$, $\mathcal X^s_{JA}(z,\Xi)\in C^\omega(D_{\C^{2n}\times \mathbf \Sigma^{-r_0}}(0,d_1)\times(-4,4),\mathbf \Sigma_c^{r_0})$ and
\begin{align}\label{eq:field0}
\sum_{J=1}^{2n}\sum_{A=R,I}|\mathcal X^s_{JA}(z,\Xi)|+\|\mathcal X^s_{\Xi}(z,\Xi)\|_{\mathbf \Sigma^{r_0}}\leq c |z|(|\mathbf Z|+\|\Xi\|_{\mathbf \Sigma^{-r_0}}).
\end{align}

  \item[\textup{(ii)}]   For $\mathcal{X}^s _{J }:=dz_J\mathcal{X}^s$ and $\mathcal{X}^s _{\Xi }= d\Xi \mathcal{X}^s$, we have
  $\mathcal{X}^s _{J } (e^{\im \theta}z, e^{\im \theta}\Xi) = e^{\im \theta}\mathcal{X}^s _{J  } (z, \Xi)$ and  $\mathcal{X}^s _{\Xi } (e^{\im \theta}z, e^{\im \theta}\Xi) = e^{\im \theta}\mathcal{X}^s _{\Xi } (z, \Xi)$.
\end{itemize}

\end{lemma}

\begin{proof}
In the following, we omit the summation.
We directly solve $i_{\mathcal{X}^s} \Omega _t=-  \Gamma$.
First,
\begin{align*}
\Omega_0(\mathcal X^s\big |Y)=\Omega(\partial_{JB}\Phi_J[z_J]\big |\partial_{JA}\Phi_J[z_J]) \mathcal X_{JB}^s Y_{JA}+ \Omega(\mathcal X_\Xi^s\big | Y_\Xi).
\end{align*}
Next, since $\Omega-\Omega_0=d \Gamma$, we have
\begin{align*}
&\Omega(\mathcal X^s\big | Y)-\Omega_0(\mathcal X^s\big |  Y)=\(\partial_{KB}\Gamma_{JA}-\partial_{JA}\Gamma_{KB}\) \mathcal X_{KB}^sY_{JA} + \<\nabla_\Xi \Gamma_{JA}\big | \mathcal X_\Xi^s\>Y_{JA}-\<\nabla_\Xi \Gamma_{KB}\big | Y_\Xi\>\mathcal X_{KB}^s\\&\quad
+\<\partial_{KB}\Gamma_\Xi\big |Y_\Xi\>\mathcal X_{KB}^s-\<\partial_{JA}\Gamma_\Xi\big |\mathcal X_\Xi^s\>Y_{JA}+\<d_\Xi \Gamma_\Xi(\mathcal X_\Xi^s)\big |Y_\Xi\>-\<d_\Xi \Gamma_\Xi (Y_\Xi)\big |\mathcal X_\Xi^s\>.
\end{align*}
Therefore, we have
\begin{align}
&2\mathbb J^{-1} \mathcal X_\Xi ^s + s\(-\mathcal X_{KB}^s\nabla_\Xi \Gamma_{KB}+\mathcal X_{KB}^s\partial_{KB}\Gamma _\Xi + d_\Xi \Gamma_\Xi(\mathcal X_\Xi^s)-(d_\Xi \Gamma_\Xi)^* \mathcal X_\Xi^s\)=-\Gamma_\Xi,\label{eq:field1}\\&
\Omega(\partial_{JB}\Phi_J[z_J]\big | \partial_{JA}\Phi_J[z_J])\mathcal X_{JB}^s+s\(\(\partial_{KB}\Gamma_{JA}-\partial_{JA}\Gamma_{KB}\) \mathcal X_{KB}^s+ \<\nabla_\Xi \Gamma_{JA}\big | \mathcal X_\Xi^s\>-\<\partial_{JA}\Gamma_\Xi\big | \mathcal X_\Xi^s\>\)\nonumber\\&=-\Gamma_{JA}.\label{eq:field2}
\end{align}
We first fix $\mathcal X_{KB}^s$ and solve \eqref{eq:field1} for $\mathcal X_\Xi^s$ by Neumann series.
Notice that the solution $\mathcal X_\Xi^s$ becomes analytic w.r.t.\ $z,\Xi,s$ and $\mathcal X_{KB}^s$.
Next, since $\Omega(\partial_{JB}\Phi_J(z_J),\partial_{JA}\Phi_J[z_J])$ is invertible, we can solve \eqref{eq:field2} again by Neumann series.
Therefore, we obtain $\mathcal X_{JA}^s$ and $\mathcal X_\Xi^s$ which satisfies \eqref{eq:field1}, \eqref{eq:field2}.
Finally, \eqref{eq:field0} follows from \eqref{eq:1forms2} in Lemma \ref{lem:1forms} and
the gauge invariance of $ \mathcal{X}$ follows from the gauge invariance of $\Omega $ and $\Omega _0$.
\end{proof}

  \begin{lemma}
  \label{lem:darflow0}    For the   $d_1>0$  and $\mathcal{X}^s  (z, \Xi ) $
   of  Lemma \ref{lem:vectorfield0},   consider
         the   following system, which  is well defined in
     $(s, z, \Xi )\in (-4, 4)\times D_{\C ^n  \times \mathbf \Sigma ^{-r_0}_c} (0,d_1 )$:
\begin{equation}\label{eq:syst1}
\begin{aligned}
  &    \frac{d}{ds} S_J = \mathcal{X}^s _{J }  (z+S_z, \Xi+S_\Xi ) \text{  and }     \frac{d}{ds} S_\Xi = \mathcal{X}^s _{\Xi}  (z+S_z, \Xi+S_\Xi  ) ,
     \end{aligned}
\end{equation}
where $S_z=(S_1,\cdots,S_{2n})$.
Then the following facts holds.

\begin{itemize}
\item[\textup{(i)}]  For   $d _2\in (0,d_1 ) $  sufficiently small system \eqref{eq:syst1}
generates flows
\begin{align} \nonumber
  &   (S_z(s,z,\Xi),S_\Xi(s,z,\Xi)) \in C ^\omega ((-2, 2)\times D_{\C ^n\times \mathbf \Sigma ^{ -r_0}_c} (0,d_2 ) ,D_{\C ^n\times \mathbf\Sigma ^{ r_0}_c} (0,d_1 )) ,
     \end{align}
with
\begin{align}\label{eq:darflow1}
\sum_{J=1}^{2n}| S_{J}(s,z,\Xi)|+\|S_{\Xi}(s,z,\Xi)\|_{\mathbf \Sigma^{r_0}}\leq c |z|(|\mathbf Z|+\|\Xi\|_{\mathbf \Sigma^{-r_0}}).
\end{align}

  \item[\textup{(ii)}]      We have  $S _{J }  (s,e^{\im \theta}z,e^{\im \theta} \Xi )=e^{\im \theta}S _{J }  (s,z, \Xi )$,
 $S _{\Xi }  (s,e^{\im \theta}z,e^{\im \theta} \Xi )=e^{\im \theta}S _{\Xi }  (s,z, \Xi )$.
 \end{itemize}

\end{lemma}

\begin{proof}
It is elementary to solve the system \eqref{eq:syst1} by contraction mapping theorem or implicit function theorem.
The properties in (i), (ii) comes from (i), (ii) of Lemma \ref{lem:vectorfield0}.
\end{proof}

From Lemmas \eqref{lem:1forms}, \ref{lem:vectorfield0} and \ref{lem:darflow0} we immediately have
Proposition \ref{prop:darboux}.

\begin{proof}[Proof of Proposition \ref{prop:darboux}]
Set $\tilde {\mathfrak F}^D(z,\Xi):=(S_z(1,z,\Xi),S_\Xi(1,z,\Xi))$.
Then, by \eqref{eq:darflow1} and (ii) of Lemma \ref{lem:darflow0}, we have \eqref{eq:propdar} and gauge property of $\mathcal F^D:=\mathrm{Id}+\tilde{\mathfrak F}^D$.
Finally, since $(S_z,S_\xi)$ is the solution of \eqref{eq:syst1} we see that $\mathfrak F^s = \mathrm{Id}+ (S_z(s,\cdot,\cdot),S_\Xi(s,\cdot,\cdot))$ is the solution of $\frac{d}{ds}\mathcal X^s(\mathfrak F^s)$.
Therefore, by \eqref{eq:argmoser}, we have $(\mathfrak F^D)^*\Omega=\Omega_0$.
\end{proof}

Consider now the    symplectic form $\Omega _0$   in \eqref{eq:newsymp}.
Notice that it is of the following form \begin{equation}\label{eq:Omega0}   \begin{aligned}          \Omega _0& = \sum_{J=1}^{2n}2\im    \widetilde{\omega }_{J} (1 + a_J(|z_J|^2)   ) dz_J
\wedge d\overline{z}_{J }  +  2
 \langle
d\Xi| \mathbb{J} d{\Xi} \rangle \text{ for }  \mathbb{J}=\begin{pmatrix}  0 & 1  \\ -1 & 0
 \end{pmatrix}
\end{aligned}\end{equation}
for appropriate real valued functions $a_J(|z_J|^2)$ with $|a_J(|z_J|^2)|\leq c|z_J|^2$.
Indeed, by direct calculation, we have
\begin{align*}
\Omega(d\Phi_J[z_J]X|d\Phi_J[z_J]Y)&=\Omega(\partial_{JR}\Phi_{J}[z_J]|\partial_{JI}\Phi_J[z_J])(X_{JR}Y_{JI}-X_{JI}Y_{JR})\\&=\frac{\im}{2}\Omega(\partial_{JR}\Phi_{J}[z_J]|\partial_{JI}\Phi_J[z_J])dz\wedge d\bar z(X,Y).
\end{align*}
Next, by Lemma \ref{lem:appendix1}, we can show $\frac 1 2 \Omega(\partial_{JR}\Phi_{J}[z_J]|\partial_{JI}\Phi_J[z_J])$ is gauge invariant.
Therefore, we can write $\frac 1 2 \Omega(\partial_{JR}\Phi_{J}[z_J]|\partial_{JI}\Phi_J[z_J])=\tilde a(|z_J|^2)$.
Finally, the constant term is given by
\begin{align*}
\frac 1 2 \Omega(\partial_{JR}\begin{pmatrix} z_J\phi_J\\ \im \tilde \omega_J z_J\phi_J \end{pmatrix}|\partial_{JI}\begin{pmatrix} z_J\phi_J\\ \im \tilde \omega_J z_J\phi_J \end{pmatrix})&=\<\begin{pmatrix} 0 & 1 \\ -1 & 0\end{pmatrix} \begin{pmatrix} \phi_J \\ \im \tilde \omega _J \phi_J \end{pmatrix}| \begin{pmatrix} \im \phi_J \\ -\tilde \omega _J \phi_J \end{pmatrix} \>\\&=2\tilde \omega_J.
\end{align*}
Therefore, we have \eqref{eq:Omega0}.

\begin{remark}
A schematic explanation of  why we adapt $\Omega_0$ and we do not use the constant symplectic form
\begin{align*}
\Omega_0':=2\im    \widetilde{\omega }_{J} dz_J
\wedge d\overline{z}_{J }  +  2
 \langle
d\Xi| \mathbb{J} d{\Xi} \rangle,
\end{align*}
is the following.
First, notice that by \eqref{eq:syst1}, $S_J\sim \mathcal X^0_J$ and $S_\Xi\sim \mathcal X^0_\Xi$.
Further, by \eqref{eq:field1} and \eqref{eq:field2}, we have $\mathcal X^0_{JA}=\mathcal A_{KB} \Gamma_{KB}$ and $\mathcal X^0_\Xi=-2^{-1}\mathbb J \Gamma_\Xi$, where $\mathcal A_{KB}$ represents the inverse of the matrix $\Omega(\partial_{JB}\Phi_J[z_J],\partial_{JA}\Phi_J[z_J])$.
Therefore, the estimate \eqref{eq:propdar} is a direct consequence of \eqref{eq:1forms2} and if \eqref{eq:1forms2} do not hold, we cannot expect \eqref{eq:propdar} to hold.
Now, if we take $\Omega_0'$ instead of $\Omega_0$, then in $\Gamma_{JA}$  there will be pure terms of $z_J$.
So we cannot have the estimate \eqref{eq:1forms2}.
Finally, we remark that the estimate \eqref{eq:propdar} is needed   not only to come back to the original coordinate after scattering but it is also crucial for the cancelation lemma (Lemma \ref{lem:KExp1} below).
\end{remark}

\noindent We introduce an index $\ell =J, \overline{J}$ 
  with $J=1,...,2n$.
Given   $F\in C^1(\mathcal{U},\R )$
with $\mathcal{U}$ an open subset of $\C^n\times \mathbf \Sigma ^{-r}_c$,  its Hamiltonian vector field $X_F$.    We have summing on $J$
\begin{equation}  \label{eq:hamf-1}\begin{aligned}              i_{X_F}\Omega _0&=    2 \im    \widetilde{\omega} _{J } (  1 +   a_J(|z_J|^2) ) ((X_F) _{J }   d\overline{z}_{J } -(X_F) _{\overline{J} }   d {z}_{J }) +
 2\langle
(X_F) _{ {\Xi}  }| \mathbb{J} d{\Xi} \rangle
\\&
 =  \partial _J F    d {z}_J   +     \partial _{\overline{J} } F  d\overline{z}_{J }    +  \<\nabla _{\Xi  } F | d {\Xi}  \> .
\end{aligned}\end{equation}

\noindent Comparing the components of the two sides of \eqref{eq:hamf-1}  for  some     $\varpi  _J(|z_J|^2) $ with $|\varpi_J(|z_J|^2)|\leq c|z_J|^2$ we get
\begin{equation} \label{eq:hamf1}\begin{aligned}        &    (X_F)_J    =-\frac 1 2 \im   \widetilde{\omega} _{J}^{-1}(1 +   \varpi  _J(|z_J|^2))
 \partial _{\overline{J}} F   \ ,        \quad     (X_F) _{\Xi}   = \frac 1 2  \mathbb{J} \nabla  _{{\Xi}  } F   .
\end{aligned}\end{equation}
 \eqref{eq:hamf1}  imply that in this latter system of coordinates equation
 \eqref{eq:NKGE}  takes the form
\begin{equation} \label{eq:hamsys1}\begin{aligned}        &   \im \dot z_J       =   \frac 1 2 \widetilde{\omega} _{J}^{-1}(1 +   \varpi  _J(|z_J|^2))
 \partial _{\overline{J}}E^D   \ ,        \quad    \dot \Xi    = - \mathbb{J} \nabla  _{\overline{\Xi}  } E^D  ,
\end{aligned}\end{equation}
where $E^D:=\(\mathfrak F^D\)^*E^0 $ which is the pull back of $E^0$ by $\mathfrak F^D$.

\section{Effective Hamiltonian}
\label{sec:nforms}

%
%
The
	pullback of the energy $E$ by the map $\mathfrak{F}   $ in    Lemma \ref{lem:darflow0}
has the following expansion.

\begin{lemma}
  \label{lem:KExp1}
For $(z,\Xi )\in D_{\C^{2n }}(0 , d  _2) \times ( D_{\mathbf H ^{\frac{1}{2}} }(0 , d  _2)\cap \mathcal{H}_c[0]) $
    we  have
\begin{align}
&E^D (z,\Xi):=(\mathfrak F^D)^*E^0(z,\Xi)   =  \sum _{J=1}^{2n}E (\Phi_J[z_J])   +   \langle B \Xi |  {\Xi} \rangle+ E _P( B^{-\frac{1}{2}}\eta)
       \nonumber
       \\& \quad
 +
	 \sum _{l =1} ^{\infty}\sum _{J =0}^{2n}	 \sum _{
	  |\textbf{m}|=l+1 }  \textbf{Z}^{\textbf{m}}   a_{J \textbf{m} }^{(1)}( |z_J |^2 )+
   \sum _{l =1} ^{\infty}\sum _{J=1}^{2n} 	 \sum _{
	  |\textbf{m}|=l  }
        \langle
  {z}_J \textbf{Z}^{\textbf{m}} G_{J\textbf{m}}^{(1)}(|z_J|^2  )| \Xi \rangle   \label{eq:enexp10}\\&  \quad
\nonumber
+ \sum _{ i+j  = 2}        \langle  G_{2 ij }^{(1)} ( z,\Xi )| (B^{-\frac{1}{2}}\eta )^{   i} \overline{B^{-\frac{1}{2}}\eta} ^j\rangle
+  \sum _{ i+j  = 3}   \langle G_{3 ij }^{(1)} ( z,\Xi )|   B^{-\frac{1}{2}}\eta ^{   i} \overline{B^{-\frac{1}{2}}\eta} ^j\rangle
\nonumber   +\resto^{(1)},
  \end{align}
where $a_{J\mathbf m}^{(1)}, G_{J\mathbf m}^{(1)}, G_{2ij}^{(1)}, G_{3ij}^{(1)}$ and $\resto^{(1)}$ satisfies the condition (1)--(3) in Lemma \ref{lem:EnExp}.

%

\end{lemma}

\begin{remark}
The only difference between \eqref{eq:enexp1} and \eqref{eq:enexp10} are that in the second line the terms with $l=0$ vanishes in \eqref{eq:enexp10}.
\end{remark}
\proof
First, we can write $E^D(z,\Xi)=E(\widetilde\Psi[z,\Xi]+B^{-\frac{\sigma_3}{2}}\Xi)$ for  $\widetilde\Psi[z,\Xi]=(\mathfrak F^D)^*\mathcal F(z,\Xi)-B^{-\frac{\sigma_3}{2}}\Xi$.
Notice that $\widetilde \Psi \in C^{\omega}(D_{\C^2n\times \mathbf \Sigma^{-r_0}}(0,d_2),\Sigma^{r_0})$.
Therefore, the proof of the expansion becomes completely parallel to the proof of \eqref{lem:EnExp}.

The only nontrivial thing remaining is the absence of the terms with $l=0$ in \eqref{eq:enexp10}.
To show this  fix $J$ and
 notice that by \eqref{eq:propdar}  a solution of \eqref{eq:NKGE} with initial value $(z,\Xi)$ (in the Darboux coordinate) with $\Xi=0$ and $z_K=0$ (for $K\neq J$) is the nonlinear bound state given in Proposition \ref{prop:bddst}.
Therefore,
if we have $\Xi(0)=0$ and $z_K(0)=0$ ($K\neq J$), we have $\Xi(t)=0$ and $z_K(t)=0$ ($K\neq J$) for all times.
In particular, we have $\left.\frac{d}{dt}\right|_{t=0} \Xi(t)=0$ and $\left.\frac{d}{dt}\right|_{t=0} z_K(t)=0$.

\noindent We first  show $z_JG_{J0}^{(1)}(|z_J|^2)=0$ for all $z_J$.
Suppose $z_JG_{J0}^{(1)}(|z_J|^2)\neq 0$ for some $z_J$.
Then, taking the initial data such as $\Xi(0)=0$, $z_K(0)=0$ ($K\neq J$) and $z_J(0)=z_J$, we have
\begin{align*}
\left.\frac{d}{dt}\right|_{t=0} \Xi(t)=-\im z_JG_{J0}(|z_J|^2)\neq 0,
\end{align*}
  contradicting $\Xi(t)=0$   for all times.

\noindent Next, we show $a_{J\mathbf m}^{(1)}(|z_J|^2)=0$ for all $|\mathbf m|=1$.
Notice that $a_{0\mathbf m}^{(1)}=0$ since $E^{0}$ has no such term by the orthogonality of $\phi_J$ and $\phi_K$ ($J\neq K$).
Furthermore, setting $\mathbf Z^{\mathbf m}=z^L\bar z^K$, we can assume that $L=J$ or $K=J$.
Indeed, suppose $L\neq J$ and  $K=J$. Notice that we have just shown $a_{J\mathbf m}^{(1)}(|z_J|^2) \left | _{z_J=0}\right . =0$
Then we can write $\mathbf Z^{\mathbf m}a_{J\mathbf m}^{(1)}(|z_J|^2)=(z_L\bar z_J)(z_J \bar z_K)(|z_J|^{-2}a_{J\mathbf m}(|z_J|^2))$ and move this    among  the terms with $|\mathbf m|=2$.

\noindent So, without loss of generality, we can assume $L=J$.
Now, suppose that $a_{J\mathbf m}(|z_J|^2)\neq 0$.
Then, taking the initial data such as $\Xi(0)=0$, $z_K(0)=0$ ($K\neq J$) and $z_J(0)=z_J$, we have
\begin{align*}
2\im\left.\frac{d}{dt}\right|_{t=0} z_K(t)=\widetilde{\omega} _{K}^{-1}z_Ja_{J\mathbf m}^{(1)}(|z_J|^2)\neq 0
\end{align*}
contradicting $z_K(t)=0$   for all times.
\qed

\bigskip

To extract an effective Hamiltonian  we cancel as many
terms as possible  from \eqref{eq:enexp10} by means of a normal forms argument.
   The following result is proved in \cite{CM1}.
\begin{proposition}[Birkhoff normal forms]
\label{th:main}  Assume (H1)--(H3).
There exists $r_B=r_B(r_0)$ and $d_B>0$ with $r_B(r_0)     \stackrel{r_0 \to \infty}{\rightarrow} \infty $ such that there exists $\tilde {\mathfrak F}^B=(\tilde {\mathfrak F}^B_z,\tilde {\mathfrak F}^B_\Xi) \in C^\omega(D_{\C^{2n}\times \mathbf \Sigma^{-r_B}}(0,d_B),\C ^{2n}\times \mathbf{ \Sigma}^{r_B})
$ such that
\begin{align}\label{eq:propbir}
|\tilde {\mathfrak F}^B_z(z,\Xi)|+\|\tilde {\mathfrak F}^B_\Xi(z,\Xi)\|_{\mathbf \Sigma^{r_B}}\leq c|z|\(|\mathbf Z|+\|\Xi\|_{\mathbf \Sigma^{-r_B}}\),
\end{align}
and $\mathfrak F^B:=\mathrm{Id}+\tilde{ \mathfrak F}^B$ satisfies $\(\mathfrak F^B\)^* \Omega_0= \Omega_0$.
In addition, we have $\mathfrak F^B(e^{\im \theta}z,e^{\im \theta}\Xi)=e^{\im \theta}\mathfrak F^B(z,\Xi)$.
Further, setting $E^B:=(\mathfrak F^B)^* E^D$, we have
\begin{align}
&E^B (z,\Xi)=     \sum _{J=1}^{2n}E (\Phi_J[z_J])   +   \langle B \Xi |  {\Xi} \rangle+ E _P( B^{-\frac{1}{2}}\eta)
       \nonumber
       \\&
 +
	  \sum _{J =0}^{2n}\sum _{
	  \textbf{m}\in \mathcal M_0(2N+4) }  	\textbf{Z}^{\textbf{m}}   a_{J \textbf{m} }( |z_J |^2 )+
   \sum _{J=1}^{2n} 	 \sum _{
	 \mathbf m \in \mathcal M(2N+3)  }
        \langle
  {z}_J \textbf{Z}^{\textbf{m}} G_{J\textbf{m}}(|z_J|^2  )| \Xi \rangle  + \resto,  \label{eq:enexp41}
\end{align}
where
\begin{align}
\resto =& \sum _{l =2N+4} ^{\infty}\sum _{J =0}^{2n}	 \sum _{
	  |\textbf{m}|=l+1 }  \textbf{Z}^{\textbf{m}}   a_{J \textbf{m} }^{}( |z_J |^2 )+
   \sum _{l =2N+4} ^{\infty}\sum _{J=1}^{2n} 	 \sum _{
	  |\textbf{m}|=l  }
        \langle
  {z}_J \textbf{Z}^{\textbf{m}} G_{J\textbf{m}}(|z_J|^2  )| \Xi \rangle  \\&
\nonumber
+ \sum _{ i+j  = 2}        \langle  G_{2 ij }^{} ( z,\Xi )| (B^{-\frac{1}{2}}\eta )^{   i} \overline{B^{-\frac{1}{2}}\eta} ^j\rangle
+  \sum _{ i+j  = 3}   \langle G_{3 ij }^{} ( z,\Xi )|   B^{-\frac{1}{2}}\eta ^{   i} \overline{B^{-\frac{1}{2}}\eta} ^j\rangle
\nonumber   +\resto^{B},
\end{align}
with $a_{J\mathbf m}$, $G_{J\mathbf m}$, $G_{2ij}$, $G_{3ij}^{}$ and $\resto^{B}$ satisfies the condition (1)--(3) in Lemma \ref{lem:EnExp}.

\end{proposition}
\qed

 We
 now make a further change of coordinates introducing $\Xi  =A\Theta $
 where  \begin{equation}\label{eq:matA}
   \text{$  \mathbb{J}   =-A \im \sigma _3  A^{-1}   $
for $A= 2  ^{-\frac{1}{2}}\begin{pmatrix}   \im   &
-\im   \\
1  &   1
 \end{pmatrix}$.}
 \end{equation}
 We have $A^*=A ^{-1}$.
So now we use $(z, \Theta )$ as coordinates. In this new coordinate system
the Hamiltonian \eqref{eq:enexp41}  takes the form

\begin{equation}  \label{eq:enexp411}    \begin{aligned} &
  \mathcal{H}  (z,\Theta ) =  \sum _{J=1} ^{2n} E(\Phi_J[ z_J])+  \langle  B \Theta |  {\Theta} \rangle  + Z  (z,\mathbf{Z},\Theta )  +\resto,  \text{ where }   Z= \mathcal{Z}_0+\mathcal{Z}_1,
\\&   \mathcal{Z}_0    = \sum _{   \mathbf{m}  \in \mathcal{M}_{0}(2N+4 )  }\sum_{J=0}^{2n} \mathbf{Z}^{\textbf{m}}   a_{ J \textbf{m} }( |z _J|^2 ),
  \\& \mathcal{Z}_1=\sum _{J =1}^{2n}	\sum _{   \mathbf{m}  \in \mathcal{M}_{J }(2N+3)  }   \langle
z_J  \mathbf{Z}^{\textbf{m}}  G _{J \textbf{m}}(|z_J |^2  )| \Theta \rangle
  = \sum _{(\mu , \nu ) \in M}\<z^{\mu}\overline{z}^{\nu}
  G _{\mu\nu}(|z_{\tilde J}|^2  )| \Theta \>   ,
     \end{aligned}
\end{equation}
where $\tilde J=\tilde J(\mu,\nu)$ is determined from the correspondence between $\mathcal M_J(2N+3)$ and $M$.

\noindent The symplectic form $\Omega _0$ takes the form \begin{equation}\label{eq:Omega00}   \begin{aligned} 2\im \widetilde{\omega} _J        (1 + a_J(|z_J|^2)   ) dz_J
\wedge d\overline{z}_{J }  +2\Im
 \langle \im
\sigma _3 d\Theta |   d{\Theta} \rangle ,
\end{aligned}\end{equation}
and system \eqref{eq:hamsys1} becomes, for  $1 +   \varpi  _J(|z_J|^2) :=  (1 + a_J(|z_J|^2) ^{-1}$,  \begin{equation} \label{eq:hamsys2}\begin{aligned}        &   \im \dot z_J       =  \frac 1 2 \widetilde{\omega} _{J}^{-1}(1 +   \varpi  _J(|z_J|^2))
 \partial _{\overline{J}}\mathcal{H}   \quad  ,  \quad    \im  \dot \Theta    = \frac 1 2  \sigma _3 \nabla  _{{\Theta}  } \mathcal{H}   .
\end{aligned}\end{equation}

The final step needed to prove Theorem \ref{thm:small en}  is the following.
\begin{proposition}[Main Estimates]\label{thm:mainbounds}
There exist $\epsilon_0>0$ and $C_0>0$ s.t.\ if the constant  $0<\epsilon   $  of Theorem \ref{thm:small en} satisfies $\epsilon<\epsilon_0$,    for $I= [0,\infty )$ and $C=C_0$    we have:
\begin{align}
&   \|  \Theta \| _{L^p_t([0,\infty ),W^{ \frac{{1}}{q}  - \frac{{1}}{p} -\frac{1}{2} ,q}_x  \times W^{ \frac{{1}}{q}  - \frac{{1}}{p} -\frac{1}{2},q}_x)}\le
  C   \epsilon \text{ for all admissible pairs $(p,q)$,}
  \label{Strichartzradiation}
\\& \| z^{\mu}\overline{z}^{\nu} \| _{L^2_t[0,\infty )}\le
  C   \epsilon \text{ for all   $(\mu , \nu ) \in M_{min}$,} \label{L^2discrete}\\& \| z _J  \|
  _{W ^{1,\infty} _t  [0,\infty  )}\le
  C   \epsilon \text{ for all   $J\in \{ 1, \dots ,  2{n}\}$ } \label{L^inftydiscrete}
   .
\end{align}
Furthermore,  there exists $\rho  _+\in [0,\infty ) ^{2n}$ s.t.   there exist  a   $j_0$  with $\rho_{+j}=0$ for $j\neq j_0$,
and   there exists $\eta _+\in H^1$   s.t.
  $| \rho  _+ - |z(0)| | \le C   \epsilon $ and $\Theta  _+\in H ^{\frac{1}{2}}\times  H ^{\frac{1}{2}}$
with $\|  \Theta   _+\| _{H ^{\frac{1}{2}}\times  H ^{\frac{1}{2}}}\le C    \epsilon $, such that
\begin{equation}\label{eq:small en31}
\begin{aligned}&     \lim_{t\to +\infty}\| \Theta (t )-
e^{\im \sigma _3t \sqrt{-\Delta +m^2} }\Theta   _+    \|_{H ^{\frac{1}{2}}\times  H ^{\frac{1}{2}}}=0  \quad  , \quad
  \lim_{t\to +\infty} |z_j(t)|  =\rho_{+j}  .
\end{aligned}
\end{equation}

\end{proposition}
\textit{Proof of Theor.\ref{thm:small en}.}
Before proving Prop.\ref{thm:mainbounds} we will show that it implies Theor.\ref{thm:small en}. The proof is similar to the analogous step in \cite{CM1}. Denote by $(z',\Xi ')$  the initial coordinate system.   By  \eqref{eq:propdar} and \eqref{eq:propbir}, we have
\begin{equation}\label{eq:diff21}  \begin{aligned}
 |z' -z|+  \|\Xi ' -A\Theta\|_{\Sigma^r} \leq c |z| \(|\mathbf Z|+\|\Theta\|_{\Sigma^{-r}}\).
\end{aligned}\end{equation}
     \eqref{eq:small en31} implies  $\lim _{t\to + \infty}\mathbf{Z}(t)=0$ and   by elementary arguments  for $s>3/2$ we have \begin{equation} \label{eq:w-conv}
\begin{aligned} &  \lim _{t\to  + \infty} \|  e^{ \im  \sigma _3 t \sqrt{-\Delta +m^2}  }\Theta  _+\| _{L^{2,-s}(\R ^3)}=0   \text{   for any $\Xi  _+\in L^2$}.
\end{aligned} \end{equation}
     These two limits,
and \eqref{eq:small en31} and \eqref{eq:diff21} imply
\begin{equation*}  \begin{aligned}
\lim_{t\to +\infty}|z'(t) -z(t)|+  \|\Xi '(t) -A\Theta(t)\|_{\Sigma^r} =0.
\end{aligned}\end{equation*}
Then we have the  limit
 \begin{equation}\label{eq:small en33}
 \lim_{t\to +\infty} |z_j'(t)| = \lim_{t\to +\infty} |z_j (t)|   =\rho  _{+j} .
\end{equation}
Next, in $H^1\times L^2$   by \eqref{eq:matA} we have
\begin{equation*}
\begin{aligned}&    0= \lim_{t\to +\infty}  B ^{-\frac{\sigma _3}{2}} [  \Xi ' (t )-   Ae^{ \im  \sigma _3 t \sqrt{-\Delta +m^2}  }\Theta  _+ ]=  \lim_{t\to +\infty}  B ^{-\frac{\sigma _3}{2}} [  \Xi ' (t )-   e^{ -\mathbb{J}  t \sqrt{-\Delta +m^2}  }\Xi _+ ]   .
\end{aligned}
\end{equation*}
Setting $\Xi _+:=A\Theta  _+$ write
\begin{equation*}
\begin{aligned}&     B ^{-\frac{\sigma _3}{2}}    e^{ -\mathbb{J}  t \sqrt{-\Delta +m^2}  }\Xi _+ =  e^{ -\mathbb{J}  t \sqrt{-\Delta +m^2}  } e^{  \mathbb{J}  t \sqrt{-\Delta +m^2}  }  e^{-  \mathbb{J}  t B  }  B ^{-\frac{\sigma _3}{2}}   e^{  \mathbb{J}  t B  }   e^{ -\mathbb{J}  t \sqrt{-\Delta +m^2}  }\Xi _+
\end{aligned}
\end{equation*}
and consider   the wave operators
\begin{equation*}
    W    = s-\lim_{t\to +\infty}    e^{  \mathbb{J}  t B  }
 e^{ -\mathbb{J}  t \sqrt{-\Delta +m^2}  }   \text{  and }  Z    = s-\lim_{t\to +\infty} e^{ \mathbb{J}  t \sqrt{-\Delta +m^2}  }   e^{ - \mathbb{J}  t B  } .
 \end{equation*}
Then we have
\begin{equation*}\begin{aligned}&    \lim_{t\to +\infty}[ B ^{-\frac{\sigma _3}{2}}    e^{ -\mathbb{J}  t \sqrt{-\Delta +m^2}  }\Xi _+ -  e^{ -\mathbb{J}  t \sqrt{-\Delta +m^2}  } Z  B ^{-\frac{\sigma _3}{2}}   W \Xi _+]=0.
\end{aligned}
\end{equation*}
Since $Z  B ^{-\frac{\sigma _3}{2}}   W \Xi _+=(-\Delta +m^2) ^{- \frac{\sigma _3}{4}} \Xi _+$
we conclude that   in $H ^{\frac{1}{2}}\times H ^{\frac{1}{2}}$
\begin{equation*}
\begin{aligned}&   \lim _{t\to \infty} [  B ^{-\frac{\sigma _3}{2}} \Xi ' (t )-
 e^{ -\mathbb{J}  t  \sqrt{-\Delta +m^2}   }    (-\Delta +m^2) ^{-\frac{\sigma _3}{4}}   \Xi _+] =0
   .
\end{aligned}
\end{equation*}
On the other hand  for
\begin{equation*}
\begin{aligned}&         u  _+   = (-\Delta +m^2) ^{ -\frac{1}{4}}   (\Xi _+)_1 \in H^1     \text{  and }  v_+=(-\Delta +m^2) ^{ \frac{1}{4}}   (\Xi _+)_2
v  _+   \in L^2   , \end{aligned}
\end{equation*}
 where $  \Xi _+=((\Xi _+)_1, (\Xi _+)_2)$,
 we have
\begin{equation*}
\begin{aligned}&   e^{ -\mathbb{J}  t  \sqrt{-\Delta +m^2}   }    (-\Delta +m^2) ^{-\frac{\sigma _3}{4}}   \Xi _+  =     \cos (t\sqrt{-\Delta +m^2})        \begin{pmatrix}     u  _+  \\
v  _+   \end{pmatrix}     - \begin{pmatrix}   -   \frac {\sin (t \sqrt{-\Delta +m^2})
} { \sqrt{-\Delta +m^2}}     v  _+  \\
  \sqrt{-\Delta +m^2}   \sin (t\sqrt{-\Delta +m^2})   u  _+   \end{pmatrix}
\end{aligned}
\end{equation*}
which yields  the 1st line in \eqref{eq:small en3}.

\noindent We have formula \eqref{eq:small en1} with $z$ replaced by $z'$ and
\begin{equation*}
\begin{aligned}&   ({\eta}   ,{\xi}  ) =  R[z']\Xi ' ={B}^{-\frac{\sigma _3}{2}}A \Theta +  \mathbf{A}\text{ where} \\& \mathbf{A} =(R[z']-{B}^{-\frac{\sigma _3}{2}})A\Theta +R[z'](\Xi'-A\Theta) .
\end{aligned}
\end{equation*}
 Then by \eqref{eq:small en31} and \eqref{eq:small en33} we conclude the $\mathbf{A}$ satisfies \eqref{eq:small en4}.

\noindent  Finally we prove the 2nd line \eqref{eq:small en2}. We have
\begin{equation*}
\begin{aligned}&          \dot z _J' +\im \omega _J z_J'  =  \dot z _J  +\im \omega _J z_J + \frac{d}{dt}\widetilde{\mathfrak F}_z(z, \Theta  ) +  \im \omega_J\widetilde{\mathfrak F}_z(z, \Theta  )  ,
\end{aligned}
\end{equation*}
where $\widetilde{\mathfrak F}:=\mathfrak F^B\circ \mathfrak F^D -\mathrm{Id}$.
We need to prove that this is $O(\epsilon ^2) $. We have   $ \dot z _J  +\im \omega _J z_J= O(\epsilon ^2) $  by \eqref{eq:FGR01} below and we have  $ \widetilde{\mathfrak F}_z(z, \Theta  ) =  O(\epsilon ^2)$ by Propositions \ref{prop:darboux} and \ref{th:main}.
For the third term, we have
\begin{equation*}
\begin{aligned}&        \frac{d}{dt}  \widetilde{\mathfrak F}_z(z, \Theta  ) =  \sum_{J=1}^{2n}\sum_{A=R,I}\partial_{JA}\widetilde{\mathfrak F}_z(z, \Theta  )\dot z_{JA} +d_\Theta\widetilde{\mathfrak F}_z(z, \Theta  ) \cdot\dot \Theta ,
\end{aligned}
\end{equation*}
with $ d _\Theta \widetilde{\mathfrak F}_z(z, \Theta  )$ the partial derivative in $\Theta$. Then, by
\begin{align*}
|\partial_{JA}\widetilde{\mathfrak F}_z(z, \Theta  )|+\|d_\Theta\widetilde{\mathfrak F}_z(z, \Theta  )\|_{\Sigma^r}\leq C (|z|+\|\Theta\|_{\Sigma^{-r}}),
\end{align*}
and equations \eqref{eq:eq f}  and \eqref{eq:FGR01}   below, we have $\frac{d}{dt}  \widetilde{\mathfrak F}_z(z, \Theta  )= O(\epsilon ^2) $.
 This yields
the inequality claimed in the  second line in  \eqref{eq:small en2}. \qed

\bigskip
By a standard argument
\eqref{Strichartzradiation}--\eqref{L^inftydiscrete} for  $I= [0,\infty )$ are a consequence of the following Proposition.

\begin{proposition}\label{prop:mainbounds} There exists  a  constant $c_0>0$  such that
for any  $C_0>c_0$ there is a value    $\epsilon _0= \epsilon _0(C_0)   $ such that   if for some $T>0$ we have
 \begin{align}
&   \|  \Theta \| _{L^p_t([0,T],W^{ {1}/{q}- {1}/{p} ,q}_x)}\le
  C _0\epsilon \text{ for all admissible pairs $(p,q)$} \label{4.4a}
\\& \| z^{\mu}\overline{z}^{\nu}\| _{L^2_t([0,T])}\le
 C_0 \epsilon \text{ for all $(\mu , \nu ) \in M_{min}$,} \label{4.4}
\end{align}
 then  \eqref{4.4a}--\eqref{4.4}
hold  for       $C=C_0/2$.
\end{proposition}

 \section{Proof of Proposition \ref{prop:mainbounds}} \label{sec:prop}

The first step in the proof is the following.
\begin{proposition}\label{Lemma:conditional4.2} Assume \eqref{4.4a}--\eqref{4.4}. Then there exist constants $c $ and $\epsilon _0>0$ s.t. for $ \epsilon \in (0, \epsilon _0)$ then we have
\begin{equation} \|  \Theta\| _{L^p_t([0,T],W^{ {1}/{q}- {1}/{p} ,q}_x)}\le
 c  \epsilon   + c    \sum _{(\mu , \nu )\in {M_{min}}  }|   z ^{\mu}\overline{z}  ^{\nu}  | _{L^2_t( 0,T  )}   \text{ for all admissible pairs $(p,q)$}\ .
  \label{4.5}
\end{equation}

\end{proposition}

\subsection{Proof of {Proposition} \ref{Lemma:conditional4.2}} \label{subsec:disp}

The facts collected in the
following lemma are proved   in \cite{bambusicuccagna}.

\begin{lemma}\label{lemma-strichartz} Assume (H1)--(H2).
\begin{itemize}

\item[\textup{(1)}]
There is a fixed constant $c_0$
  such that for any two admissible
pairs $( p,q)$ and $(a,b)$ we have
\begin{equation} \label{strichartz2}\begin{aligned} & \|   e^{\pm \im tB}P_c u_0 \|
_{L^p (\R , W^{\frac{1}{q}-\frac{1}{p} ,q} (\R ^3))}\le
 c_0\|  u_0   \| _{ H^{1/2}(\R ^3)}
\\
&     \|  \int _{s<t}  e^{\pm  \im (s-t)B}P_c F(s)ds \|
_{L^p(\R ,W^{\frac{1}{q}-\frac{1}{p} ,q}(\R ^3))}\le
 c _0\| F \| _{L^{a'}(\R ,W^{\frac{1}{a}-
 \frac{1}{b}+1, b'} (\R ^3))}
.
\end{aligned}
\end{equation}

\item[\textup{(2)}] For $u_0 \in H^{2,
s}(\R ^3, \C )$ for $s>1/2$ and    $\lambda >m$ then   $\displaystyle R^{\pm}_B(\lambda
)u_0  := \lim _{\varepsilon \to 0^{\pm}}  R_B (\lambda +\im \varepsilon )u_0$ is well defined in  $L^{2, -s}(\R ^3, \C )$.

\item[\textup{(3)}] For any $s>1$
there is a fixed $c_0=c_0(s,a)$ such that for any admissible pair
$(p,q)$ we have
\begin{equation}\label{eq:christkiselev}
\left\|  \int _{0} ^t e^{\pm \im (t'-t)B}P_cF(t') dt'\right \|
_{L^p(\R ,W^{\frac{1}{q}-\frac{1}{p} ,q} (\R ^3))} \le c_0    \| B^{\frac
12}P_cF\|_{L ^a(\R ,L ^{ 2, s } (\R ^3))}
\end{equation}

where for $p>2$ we
can pick any $a\in [1,2]$ while for $p=2$ we pick $a\in [1,2)$.

\end{itemize}

\end{lemma}

\qed

 Now we look at the equation for $\Theta$.   Then  for ${G}_{\mu\nu}  = {{G}}_{\mu\nu} (0  )$
 we have
\begin{equation}  \label{eq:eq f}  \begin{aligned} &
    \im \dot \Theta= \frac 1 2 \sigma_3  \nabla_{\Theta}  \mathcal{H}=  \sigma_3  B \Theta +\frac 1 2 \sigma_3
      \sum _{  (\mu , \nu)\in M _{min}  } {z}^{\mu}\bar z^\nu  {{G}}_{\mu\nu}  + \mathbb{A}  \text{  where}\\&
    \mathbb{A}:=
   \frac 1 2 \sigma_3\(\sum _{  (\mu , \nu)\in M _{min}  } {z}^{\mu}\overline z^\nu
  [ {{G}}_{\mu\nu} (|z_{\tilde J}|^2  )  -  {{G}}_{\mu\nu}  ] +\sum _{  (\mu , \nu)\in M\backslash M _{min}  } {z}^{\mu}\overline z^\nu  {{G}}_{\mu\nu} + \nabla _{ \Theta} \resto\)   .
\end{aligned}\end{equation}
 The following lemma is proved in Lemma 7.5 \cite {bambusicuccagna}.
\begin{lemma}\label{lemma:bound remainder}
Assume \eqref{4.4a}--\eqref{4.4}, and fix  a large $s>0$.  Then there is a   constant $C=C(C_0)$ independent of $\epsilon$ such that the following is true: we have $ \mathbb{A} =R_1+ R_2$ with
\begin{equation} \label{bound1:z1} \| R_1 \| _{L^1 ([0,T],H^{\frac{1}{2} }(\R ^3))}+\|B^{\frac 12}P_c R_2 \| _{L^{2
\frac{N+1}{N+2}} ([0,T],L^{2,s}(\R ^3))}\le C(C_0) \epsilon^2.
\end{equation}
\end{lemma}
\qed

\textit{Proof of {Proposition} \ref{Lemma:conditional4.2}.}
 Using Lemma \ref{lemma:bound remainder}
 we write
\begin{equation*}\label{duhamel:f}
\Theta = e^{-\im \sigma _3 Bt}\Theta (0)
       -\frac 1 2 \im \sigma_3 \sum _{(\mu , \nu )\in {M}_{min}  }   \int _0^t  z^{\mu} \overline{z}^{\nu}e^{ -\im \sigma _3 B(s-t)} {G}_{\mu\nu} ds - \im \sum _{j=1}^{2}\int _0^t   e^{ \im \sigma _3 B(s-t)} P_c R_j ds.
\end{equation*}
By \eqref{strichartz2} for $(a,b)=(\infty ,2)$  and \eqref{bound1:z1}
\begin{equation*}  \label{eq:cond 1} \begin{aligned} &
 \|  \int _0^t  e^{ \im \sigma _3 B(s-t)}
    R_1 ds  \| _{L^p ([0,T],W^{ \frac {1} {q}- \frac{1} {p} ,q}(\R ^3))}
    \le C \| R_1 \| _{L^1_t([0,T],H^{\frac{1}{2} }(\R ^3))}
    \le C(C_0) \epsilon^2.
\end{aligned}
\end{equation*}
By \eqref{eq:christkiselev} and \eqref{bound1:z1}, we get
for $s>1$
\begin{equation*}  \label{eq:cond 2} \begin{aligned} &
 \|  \int _0^t  e^{ \im \sigma _3 B(s-t)}
   P_c R_2 ds  \| _{L^p_t([0,T],W^{ \frac {1} {q}- \frac{1} {p}  ,q}(\R ^3))}
      \le   {C} \|  \sqrt{B} P_cR_2 \| _{L^{2\frac{N+1}{N+2}}_t([0,T],
  L^{2, s }(\R ^3))}    \le C(C_0) \epsilon^2.
\end{aligned}
\end{equation*}
Finally, we have
\begin{align*}
\|\int _0^t  z^{\mu} \overline{z}^{\nu}e^{ -\im \sigma _3 B(s-t)} {G}_{\mu\nu}   \|
_{L^p([0,T],W^{\frac{1}{q}-\frac{1}{p} ,q}(\R ^3))}\leq &C_0 \| z^{\mu} \overline{z}^{\nu}  {G}_{\mu\nu} \| _{L^2([0,T],W^{\frac{1}{3}+\frac{1}{2}  , \frac{6}{5}}(\R ^3))}\\ \leq &c_0\| z^{\mu} \overline{z}^{\nu}    \| _{L^2 [0,T] }.
\end{align*}
\qed

Now, following a    standard argument, we  introduce a new variable $g$ setting
\begin{equation} \label{eq:def g} \begin{aligned} &
 g   =  \Theta +Y \text{   with } Y:=
  -\frac 1 2\sum _{(\alpha , \beta )\in M _{min}  }   {{z}}^\alpha \overline {{z}}^\beta   R _{ -\sigma _3 B}^{+}( \widetilde{{\omega }}     \cdot (\beta-\alpha  ))
  \sigma _3 {G}_{\alpha \beta  }     .
\end{aligned}\end{equation}
Substituting \eqref{eq:def g}  in \eqref{eq:eq f}   we obtain
\begin{equation} \label{eq:eq g1}  \begin{aligned} &
 \im \dot g =     \sigma_3  B g  +   \im \dot Y - \sigma_3  BY  +\frac 1 2 \sigma_3
      \sum _{  (\mu , \nu)\in M _{min}  } {z}^{\mu}\overline z^\nu  {{G}}_{\mu\nu}  + \mathbb{A},
\end{aligned}\end{equation}
where $\widetilde{\omega}=(\widetilde{ \omega} _{1},...,\widetilde{ \omega} _{2n}) $
has been defined in \eqref{eq:veceig}.
We then compute
\begin{equation*} \label{eq:eq g2}  \begin{aligned} &
   \im \dot Y = \sigma_3B Y  -\frac 1 2 \sigma_3
      \sum _{  (\mu , \nu)\in M _{min}  } {z}^{\mu}\overline z^\nu  {{G}}_{\mu\nu}   + \mathbf{T}, \text{ where} \\&     \textbf{T} :=\sum _{ J=1}^{2n} \left [\partial _{z_J}Y (\im \dot z_J- \widetilde{\omega}_Jz_J)+\partial _{\overline{z}_J}Y (\im \dot {\overline{z}}_J+ \widetilde{\omega}_J\overline{z}_J)
 \right ] .
\end{aligned}\end{equation*}
Then \begin{equation} \label{eq:eq g3}  \begin{aligned} &
 \im \dot g =     \sigma_3  B g  + \mathbf{T}  + \mathbb{A}
\end{aligned}\end{equation}
where $ \mathbf{T}$  is smaller than the quantities which have been canceled out.
The following result  is proved in \cite{bambusicuccagna}.
\begin{lemma}\label{lem:bound g} Assume \eqref{4.4a}--\eqref{4.4}. Fix $s>9/2$. Then, there are constants $\epsilon _0>0$
and $C>0$ such that, for $\epsilon \in (0,\epsilon _0) $ and   for
$c_0$ the constant in Lemma \ref{lemma-strichartz}, we have
\begin{equation} \label{bound:auxiliary}\| g
\| _{L^2_t([0,T],H^{-4,-s} (\R ^3))}\le c_0 \epsilon + C\epsilon ^2
.\end{equation}
\end{lemma}\qed

\subsection{Estimate of the discrete variables $z $}
\label{discrete}

In the following, we use $
    \(   g |h\) =
 \int_{\R^3} \(g_1(x) \overline{h_1(x)}+g_2(x) \overline{h_2(x)}\)\,dx$.

Let us turn now to the analysis of the Fermi Golden Rule (FGR).
The equation of $z$ is
 \begin{align}
\im \dot z _J
&= (2\widetilde{\omega} _{J })^{-1}( 1 +   \varpi  _J(|z_J|^2))
  \Biggl\{\(  \partial _{\overline{z}_J}  E(\Phi_J[z_J]+
\partial _{\overline{z}_J} {\mathcal{Z}}_0(z)     \right) \Biggr.\label{eq:FGR01}\\&\quad\quad+  \Biggl.\frac 1 2
 \sum _{ (\mu , \nu )\in M_{min}  }     \left [  \frac{z  ^{\mu }
 \overline{ {z }}^ { {\nu}  } }{\overline{z}_J}
\nu _J\(
  {G}_{\mu \nu }  | \Theta\)      +     \frac{\bar z  ^{\mu }
 { {z }}^ { {\nu}  } }{\overline{z}_J}
  \mu _J
\(
 {\overline{{G  }}}_{\mu  \nu  } | \overline{\Theta }\) \right ]  +   \mathfrak{E}_J\Biggr\},\nonumber
\end{align}
  and
\begin{align*}
\\ \mathfrak{E} _J &:=
 \frac 1 2
 \sum _{ (\mu , \nu )\in M\backslash M_{min} }
\left [ \frac{z  ^{\mu }
 \overline{ {z }}^ { {\nu}  } }{\overline{z}_J} \nu _J\(
  {G}_{\mu \nu } (|z_{\tilde J}|^2)  |\Theta\)      +     \frac{\bar z  ^{\mu }
 { {z }}^ { {\nu}  } }{\overline{z}_J}  \mu _J
\(
 {\overline{{G  }}}_{\mu  \nu  }(|z_{\tilde J}|^2) | \overline{\Theta } \) \right ] +
\partial _{  \overline{z} _J}   \mathcal{R} \nonumber \\  &
+      \frac 1 2 \sum _{ (\mu , \nu )\in M_{min} }    \nu _J \frac{z  ^{\mu }
 \overline{ {z }}^ { {\nu}  } }{\overline{z}_J}
\(
 {G}_{\mu \nu }(|z_{\tilde J}|^2)- {G}_{\mu \nu } | \Theta \)      + \frac 1 2 \sum _{(\mu  , \nu  )\in M_{min}}    \mu _J   \frac{z  ^{\nu  }
 \overline{ {z }}^ { {\mu}   } }{\overline{z}_J}
\(
 \overline{{G}}_{\mu \nu }(|z_{\tilde J}|^2)-{\overline{{G  }}}_{\mu  \nu  }  | \overline{\Theta }\) \nonumber    \\&
 +   \sum_{(\mu , \nu )\in M}  \delta_{J\tilde J} z_J\<z^\mu\bar z^\nu G_{\mu\nu}'(|z_{\tilde J}|^2)|\Theta\>.
\end{align*}
Estimates  \eqref{4.4a}--\eqref{4.4} imply the following simple lemma, see \cite{CM1}
for the proof.
\begin{lemma}\label{lem:estfgr1}  For $\epsilon _0>0$ sufficiently small
there is a fixed $c_1$ s.t. for $\epsilon\in (0,\epsilon _0) $ we have
 \begin{equation}\label{eq:estfgr11}
  \begin{aligned}  & \| \mathfrak{E}_J z_J\| _{L^1 [0,T]}\le   c_1 \epsilon ^2   .
\end{aligned}\end{equation}
\end{lemma} \qed

Now, we substitute $ \Theta = g-Y$ from \eqref{eq:def g} obtaining
\begin{equation}\label{eq:FGR02}
\begin{aligned}
\im \dot z _J
&=\frac{ 1 +   \varpi  _J(|z_J|^2)}{2\widetilde{\omega} _{J }}
\Biggl\{  \partial _{\overline{z}_J}  E(\Phi_J[z_J]+
\partial _{\overline{z}_J} {\mathcal{Z}}_0  (z)  \Biggr.  \\&
+\frac1 4\sum _{ \substack{  (\mu , \nu )\in M _{min}\\
(\alpha , \beta )\in M_{min}}}  \left [ \nu _J  \frac{z  ^{\mu +\beta}
 \overline{ {z }}^ { {\nu} +\alpha } }{\overline{z}_J}
\(G_{\mu\nu}|  R_  {-\sigma _3B}^{+}( \widetilde\omega     \cdot (\beta -
\alpha )) \sigma _3 {G}_{\alpha \beta}   \)  \right .
\\&\Biggl.
+  \left .       \mu _J   \frac{z  ^{\nu  + \alpha  }
 \overline{ {z }}^ { {\mu} +\beta    } }{\overline{z}_J}
\(\overline{ G}_{\mu \nu}|  R_  {-\sigma _3B}^{-}(\widetilde\omega    \cdot (\beta  -
\alpha  )) \sigma _3 \overline{G}_{\alpha   \beta  }    \)  \right ]    +
    \mathcal{F} _J \Biggr\} \text{, where}
\end{aligned}
\end{equation}
 \begin{align} &
 \mathcal{F} _J :=\frac{ 1 +   \varpi  _J(|z_J|^2)}{4\widetilde{\omega} _{J }}
 \left \{\sum _{ (\mu , \nu )\in M_{min}  }
\left [ \frac{z  ^{\mu }
 \overline{ {z }}^ { {\nu}  } }{\overline{z}_J} \nu _J\(
  {G}_{\mu \nu } | g \)      +     \frac{\bar z  ^{\mu }
 { {z }}^ { {\nu}  } }{\overline{z}_J}   \mu _J
\(
 {\overline{{G  }}}_{\mu  \nu  } | \overline{g } \) \right ]  +   \mathfrak{E}_J \right \}  .\nonumber
\end{align}
From \eqref{bound:auxiliary}  and \eqref{eq:estfgr11}  we get the following simple bound.
\begin{lemma}\label{lem:estfgr2}  For $\epsilon _0>0$ sufficiently small
there is a fixed $c_1$ s.t. for $\epsilon\in (0,\epsilon _0) $ we have
 \begin{equation}\label{eq:estfgr111}
  \begin{aligned}  &  \|\mathcal{F} _J z_J \| _{L^1 [0,T]}\le (1+C_0) c_1 \epsilon ^2   .
\end{aligned}\end{equation}
\end{lemma}
\qed

We now introduce a  new variable   $\zeta$ defined by  $\zeta _J = z _J  + {T}_J (z)$  where
\begin{equation}\label{eq:FGR21}  \begin{aligned}   &
  {T}_J (z):= -\frac 1 8
 \sum _{ \substack{  (\mu , \nu )\in M_{min} \\
(\alpha , \beta )\in M_{min}}}       \frac{\nu _Jz  ^{\mu +\beta}
 \overline{ {z }}^ { {\nu} +\alpha } }{((\mu - \nu)\cdot \widetilde{\omega }-(\alpha - \beta )\cdot \widetilde{\omega }  )\overline{z}_J}
\(  {G}_{\mu \nu } | R_{-\sigma _3B}^{+}( \widetilde{\omega }     \cdot (\beta -
\alpha ))   \sigma _3{{G} }_{\alpha \beta}
   \)  \\&
   - \frac 1 8  \sum _{ \substack{  (\mu , \nu )\in M_{min} \\
(\alpha , \beta )\in M_{min}}}         \frac{\mu _J  z  ^{\nu + \alpha }
 \overline{ {z }}^ { {\mu}+\beta  } }{((\alpha  - \beta )\cdot \widetilde{\omega} -(\mu  - \nu  )\cdot  \widetilde{\omega }  )\overline{z}_J}
\(
  \overline{{G}} _{\mu \nu }  | R_{-\sigma _3B}^{-}(  \widetilde{\omega }    \cdot (\beta -
\alpha )) \sigma _3 \overline{G}_{\alpha  \beta } \),
  \end{aligned}
\end{equation}
with the summation   performed over the pairs where the formula makes sense, that is $(\mu -\nu )\cdot \widetilde{\omega}\neq (\alpha -\beta )\cdot \widetilde{\omega} .$
It is easy to see that
\begin{equation}  \label{equation:FGR3} \begin{aligned}   & \| \zeta  -
 z  \| _{L^2(0,T)} \le  c(N,C_0) \epsilon ^2  \text{  and }  \| \zeta  -
 z \| _{L^\infty (0,T)} \le  c(N,C_0) \epsilon ^2 .
\end{aligned}
\end{equation}
We set $\Lambda := \{ (\nu - \mu ) \cdot \widetilde{\omega}    :   (\mu  , \nu )\in M_{min} \}      $.
For any $L\in \Lambda $ set
\begin{equation}  \label{eq:FGR23}    \begin{aligned}   &   M_L  := \{  (\mu  , \nu )\in M _{min}: (\nu - \mu ) \cdot \widetilde{\omega}   =L  \}  .
\end{aligned}    \end{equation}
Then all terms of
\eqref{eq:FGR02}   where $(\mu -\nu )\cdot \widetilde{\omega}\neq (\alpha -\beta )\cdot \widetilde{\omega}  $
cancel out giving an equation like \eqref{eq:FGR251}  below,  which is the equation satisfied
by $\zeta$ and involves an error term
which by \cite{CM1}   satisfies \eqref{4.33}.

\begin{lemma}\label{lem:chcoo}    $\zeta$ satisfies  equations
 \begin{align} \nonumber
 &\im \dot \zeta _J
 =\frac{ 1 +   \varpi  _J(|\zeta _J|^2)}{2\widetilde{\omega} _{J }}
 \Biggl\{  \partial _{\overline{\zeta}_J}  E(\Phi_J[\zeta _J]+
\partial _{\overline{\zeta }_J} {\mathcal{Z}}_0  (\zeta )   +\\&
\frac{1}{4} \sum _{L\in \Lambda}  \sum _{ \substack{  (\mu , \nu )\in M _L\\
(\alpha , \beta )\in M _L }}  \( \nu _J  \frac{\zeta  ^{\mu +\beta}
 \overline{ \zeta}^ { {\nu} +\alpha } }{\overline{\zeta}_J}
\( {G}_{\mu \nu }  | R_  {-\sigma _3B}^{+}(L)  \sigma _3{G}_{\alpha \beta}
   \)
+       \mu _J   \frac{\zeta  ^{\nu  + \alpha  }
 \overline{ \zeta}^ { {\mu} +\beta    } }{\overline{\zeta}_J}
\( \overline{{G}}_{\mu  \nu  } | R_  {-\sigma _3B}^{-}(L) \sigma _3 \overline{G}_{\alpha   \beta  }
    \)  \) \nonumber
\\&  \quad \quad \quad \quad \quad \quad \quad \quad \quad  \Biggl.+
    \mathcal{G} _J  \Biggr\},   \label{eq:FGR251}
\end{align}
where
there are fixed $c_1$  and $ \epsilon _0>0$ such that
 for $T>0$ and $\epsilon \in (0, \epsilon _0)$ we have
 \begin{equation}\label{4.33}
  \begin{aligned}  & \sum _{J=1}^{2n} |\widetilde \omega| \| \mathcal{G}_J   \zeta _J\| _{L^1 [0,T]}\le (1+C_0) c_1 \epsilon ^2   .
\end{aligned}
\end{equation}

\end{lemma}
\noindent We multiply \eqref{eq:FGR251} by $\overline{\zeta} _J\widetilde \omega_J^2 (1 + a_J(|\zeta _J|^2)   ) $, see \eqref{eq:Omega00}.  Let
\begin{equation*}
   A_J (s) :=4\int _0^s \widetilde \omega_J^2 \(1 + a_J(s)\)  ds   .
\end{equation*}
Near 0,
since $a_J(s)=O(s )$, we have $A_J (s)=\widetilde \omega_J^2s + O(s^2)$.
  By \eqref{eq:FGR251}, we obtain
  \begin{align*}
&  \frac{1}{4}\sum_{J=1}^{2n} \frac{d}{dt}  A_J ( |\zeta_J|^2) =  \sum_{J=1}^{2n} \Im [\widetilde \omega_J    \partial_{\overline{\zeta} _j} E(\Phi_J[\zeta _J]) \overline{\zeta} _J ] + \sum_{J=1}^{2n} \Im [  \widetilde \omega_J
\partial _{\overline{\zeta }_J} {\mathcal{Z}}_0  (\zeta )   ]
\\& +  \sum_{J=1}^{2n}\(\frac{\widetilde \omega_J}{4}\sum _{L\in \Lambda}  \sum _{ \substack{  (\mu , \nu )\in M _L\\
(\alpha , \beta )\in M _L }}  \Im\( (\nu _J-\mu_J)  \zeta  ^{\mu +\beta}
 \overline{ \zeta}^ { {\nu} +\alpha }
\( {G}_{\mu \nu }  | R_  {-\sigma _3B}^{+}(L)  \sigma _3{G}_{\alpha \beta}
   \)
  \) +\tilde \omega_J\Im \({\mathcal {G}}_J\bar \zeta_J\)\).
\end{align*}
The 1st    in the r.h.s.  equals 0  since  $E(\Phi_J[\zeta _J]) =F_J ( |\zeta_J|^2)$
for $F_J$ real valued. Also the 2nd term in the r.h.s. equals 0.
Indeed,
\begin{equation*}   \begin{aligned} &   \sum _{J }
  \widetilde \omega_J ( \overline{z} _J \partial _{\overline{z }_J} \mathcal{Z}_0   -z _J \partial _{ {z }_J} \mathcal{Z}_0   ) =\\&   \sum _{   \mathbf{m}  ,K  } a_{ K \textbf{m} }( |z _K|^2 )
  \sum _{J }
  \widetilde \omega_J ( \overline{z} _J \partial _{\overline{z }_J} \mathbf{Z}^{\textbf{m}}   -z _J \partial _{ {z }_J} \mathbf{Z}^{\textbf{m}}   ) +
\sum _{J } \sum _{   \mathbf{m}    } \mathbf{Z}^{\textbf{m}}   a_{ J \textbf{m} }'( |z _J|^2 ) |z _J|^2  .
\end{aligned}\end{equation*}
The 2nd term in the r.h.s. is real valued. The 1st term in the r.h.s. is equal to
\begin{equation*}   \begin{aligned} &      \sum _{   \mathbf{m}  ,K  } a_{ K \textbf{m} }( |z _K|^2 )
 \sum _{J }
  \widetilde \omega_J ( \overline{z} _J \partial _{\overline{z }_J} \mathbf{Z}^{\textbf{m}}   -z _J \partial _{ {z }_J} \mathbf{Z}^{\textbf{m}}   )     =  \sum _{   \mathbf{m}  ,K  } a_{ K \textbf{m} }( |z _K|^2 )\mathbf{Z}^{\textbf{m}}\sum _{J,L}    \widetilde{\omega}_J ( m _{LJ}  - m _{JL} ) =0.
\end{aligned}\end{equation*}
  where we have used  \eqref{eq:setM0} for the last equality.

  \noindent
For
\begin{equation} \label{eq:pos2} \begin{aligned} &
  G _L(\zeta ) :=   \sum _{   (\mu , \nu )\in M _L } \zeta  ^{\mu }
 \overline{ {\zeta }}^ { {\nu}   } {G}_{\mu \nu }
\end{aligned}\end{equation}
 we conclude
\begin{align*}
   \sum_{J=1}^{2n} \frac{d}{dt}  A_J ( |\zeta_J|^2) &=
\frac{1}{4}\sum _{L\in \Lambda}  \widetilde \omega\cdot(\nu -\mu) \Im
\( {G}_{L }  | R_  {-\sigma _3B}^{+}(L)  \sigma _3{G}_{L}
  \) +\sum_{J=1}^{2n}\widetilde \omega_J\Im \({\mathcal G}_J\bar \zeta_J\)\\&=\frac{1}{4}\sum _{L\in \Lambda}  L \Im\(
 {G}_{L }  | R_  {-\sigma _3B}^{+}(L)  \sigma _3{G}_{L}
   \)
   +\sum_{J=1}^{2n}\widetilde \omega_J\Im \({\mathcal G}_J\bar \zeta_J\) .
\end{align*}
We can now substitute
$
  R_{-\sigma _3B}^{+ }( L) = P.V. \frac{1}{{-\sigma _3B}-L}+ \im \pi \delta ( -{\sigma _3B}-L) ,$  which can be defined in terms of distorted Fourier transform associated to $-\Delta +V$, see \cite{taylor}.  The contribution of   the principal value cancels out because $P.V. \frac{1}{{-\sigma _3B}-L}\sigma_3$ is symmetric.
Therefore, we have
\begin{align}
& \sum_{J=1}^{2n} \frac{d}{dt}  A_J ( |\zeta_J|^2) =  -\frac{\pi}{4}\sum _{L\in \Lambda}  L \<
 {G}_{L }  | \delta(-\sigma_3 B-L)  \sigma _3{G}_{L}
   \>
   +\sum_{J=1}^{2n}\widetilde \omega_J\Im \({\mathcal G}_J\bar \zeta_J\)\label{eq:master0}\\&=
\frac{\pi}{4}\(\sum _{\substack{L\in \Lambda\\ L>m}}  L \<
 \pi_2{G}_{L }  | \delta(B-L)   \pi_2{G}_{L}
   \> +\sum _{\substack{L\in \Lambda\\ L<-m}}  |L| \<
 \pi_1{G}_{L }  | \delta(B-|L|)   \pi_1{G}_{L}
   \>  \)
   +\sum_{J=1}^{2n}\widetilde \omega_J\Im \({\mathcal G}_J\bar \zeta_J\),\nonumber
\end{align}
where $\pi_j$ is the projection to the $j-th$ component and we have used
 \begin{equation*}   \begin{aligned} &  \delta ({-\sigma _3B-L})  \sigma _3    =  \left\{\begin{matrix}   \begin{pmatrix}  0 &
0 \\
0 &  - \delta ({ B-L})
\end{pmatrix}
  \text{ if $L>m$ ,}   \\   \begin{pmatrix}    \delta ({- B-L}) &
0 \\
0 &  0
\end{pmatrix}
 \text{ if $L<-m$,}
\end{matrix}\right.
\end{aligned}\end{equation*}
and $\delta(-B-L)=\delta(B+L)$.
We claim
\begin{equation}\label{claim1}
      \langle \delta ({B-|L|}) u  |
    u  \rangle  \geq 0.
\end{equation}
Let $\kappa\in \R$, $\tau \in \R_+$ and $f:\R _+\to \R$ smooth strictly monotonic and with
$f(\tau ^2)=\kappa $.
 Then for $H= (-\Delta +V)P_c$, and for $\widehat{u}$ and  $\widehat{v}$  distorted Fourier transforms
 associated to $H$ we have,   ch.2 \cite{friedlander},
\begin{equation*}   \begin{aligned} &    \langle \delta ( f(H)-\kappa ) u |  {v}
\rangle   =   \langle    \delta (  f(\xi ^2)-\kappa ) \widehat{u}  |
 {\widehat{v}}  \rangle  =\int _{\R} dt \delta (  t-\tau ) \int _{f(\xi ^2)=t} \widehat{u}(\xi ) \overline{\widehat{v}}(\xi ) \frac{dS_t}{2 f'(t ^2)t} \\& = \frac{1}{2 f'(\tau ^2)\tau} \int _{|\xi |=\tau }  \widehat{u}(\xi ) \overline{\widehat{v}}(\xi )dS   \text{ with }dS=dS_\tau ,
\end{aligned}\end{equation*}
with $dS_t$  the standard measure on the  sphere   $|\xi |=t$  in $\R ^3$.
\noindent For   $L>m$ and $f(t)=\sqrt{t+m^2} $ we have
\begin{equation*}   \begin{aligned} &    \langle \delta ( B-|L| ) u|  u
\rangle   =    \frac{|L|}{\sqrt{L^2-m^2}} \int _{|\xi |=\sqrt{L ^2-m^2} } |\widehat{u}| ^2 dS\ge 0.
\end{aligned}\end{equation*}
We are are finally able to state in  a precise way  hypothesis (H4).

\begin{itemize}
\item[(H4)]  We assume
that there is a fixed constant $\mathfrak{c} >0$ s.t. for all $\zeta \in
\mathbb{C} ^{2n}$  with $|\zeta| \le 1$,
\begin{equation}\label{eq:FGR}
\sum _{\substack{L\in \Lambda\\ L>m}}  L \<
 \pi_2{G}_{L }  | \delta(B-L)   \pi_2{G}_{L}
   \> +\sum _{\substack{L\in \Lambda\\ L<-m}}  |L| \<
 \pi_1{G}_{L }  | \delta(B-|L|)   \pi_1{G}_{L}
   \>
\geq \frac {4\mathfrak{c} }{\pi}\sum_{(\mu,\nu)\in M_{min}}|\zeta^{\mu+\nu}|^2.
\end{equation}
\end{itemize}

By an application of Lemma \ref{lem:chcoo}, \eqref{eq:master0} and (H4)
we arrive  as in \cite{CM1} at
\begin{align}
 \sum_{J=1}^{2n} \frac{d}{dt}  A_J ( |\zeta_J|^2) \geq
\mathfrak c \sum_{(\mu,\nu)\in M_{min}}|\zeta^{\mu+\nu}|^2
   +\sum_{J=1}^{2n}\widetilde \omega_J\Im \(\mathcal G_J\bar \zeta_J\) .\nonumber
\end{align}
Using $A_J ( |\zeta_J (t)|^2) =   4\omega_J^2|\zeta_J (t)|^2 + O(|\zeta_J (t)|^4)$,
\eqref{equation:FGR3}, going back to the   $z$ and for $c_1$ the constant in   \eqref{4.33},
when we integrate the above equation, we get
\begin{equation*}  \begin{aligned} &
     4  \sum _{J=1} ^{2}    \omega _J^2 (|  z_J(0)  | ^2 - |  z_J(t)  | ^2) + \mathfrak{c}\sum _{ ( \mu , \nu )
   \in M _{min}}  \| z^{\mu +\nu } \| _{L^2(0,t)}^2
 \le   3 (1+C_0) c_1 \epsilon ^2   .
\end{aligned}
\end{equation*}
On the other hand, by the conservation of the energy \eqref{eq:enexp411}  we have $|z(t)|\lesssim \epsilon $ for a some fixed constant. So we can conclude, perhaps for a larger $c_1$, that
\begin{equation*} \label{eq:crunch}\begin{aligned}&  \sum _{ ( \mu , \nu )
   \in M}  \| z^{\mu +\nu } \| _{L^2(0,t)}^2\le  8 \mathfrak{c} ^{-1}c_1 (1+C_0)  \epsilon ^2.\end{aligned}
\end{equation*}
This  tells us that $\| z ^{\mu +\nu } \| _{L^2(0,t)}^2\lesssim
  C_0^2\epsilon ^2$ implies  $\| z ^{\mu +\nu } \|
_{L^2(0,t)}^2\lesssim \epsilon ^2+ C_0\epsilon ^2$ for all $( \mu , \nu )
   \in M _{min}$.   This means that we can take
$C_0\sim 1$  and completes the proof of  Prop. \ref{prop:mainbounds}.

\appendix

\section{Proof of (iii) of Lemma \ref{lem:contcoo}}
\begin{lemma}\label{lem:appendix1}
Let $f:\C\to \C$ satisfy $f(e^{\im \theta}z)=e^{\im \theta}f(z)$.
Then, we have
\begin{align*}
e^{-\im \theta}\partial_{z_R}f(e^{\im \theta}z)&=\cos \theta \partial_{z_R}f(z)-\sin\theta \partial_{z_I}f(z),\\
e^{-\im \theta}\partial_{z_I}f(e^{\im \theta}z)&=\sin \theta \partial_{z_R}f(z)+\cos\theta \partial_{z_I}f(z).
\end{align*}
\end{lemma}

\begin{proof}
We set $f_\theta(z):=f(e^{\im \theta}z)$.
First, we have
\begin{align*}
\partial_{z_R}f_\theta(z)&=\left.\frac{d}{d \varepsilon}\right|_{\varepsilon=0}f(e^{\im \theta}(z+\varepsilon))=\left.\frac{d}{d \varepsilon}\right|_{\varepsilon=0}f(e^{\im \theta}z+\cos\theta \varepsilon+\im \sin \theta \varepsilon)\\&=\cos\theta \ \partial_{z_R}f(e^{\im \theta}z)+\sin\theta \ \partial_{z_I}f(e^{\im \theta}z) ,\\
\partial_{z_I}f_\theta(z)&=\left.\frac{d}{d \varepsilon}\right|_{\varepsilon=0}f(e^{\im \theta}(z+\im \varepsilon))=\left.\frac{d}{d \varepsilon}\right|_{\varepsilon=0}f(e^{\im \theta}z-\sin\theta \varepsilon+\im \cos \theta \varepsilon)\\&=\sin\theta \ \partial_{z_R}f(e^{\im \theta}z)+\cos\theta \ \partial_{z_I}f(e^{\im \theta}z).
\end{align*}
On the other hand,
\begin{align*}
\partial_{z_R}f_\theta(z)&=\left.\frac{d}{d \varepsilon}\right|_{\varepsilon=0}f(e^{\im \theta}(z+\varepsilon)=e^{\im \theta}\left.\frac{d}{d \varepsilon}\right|_{\varepsilon=0}f(z+\varepsilon)=e^{\im \theta}\partial_{z_R}f(z),\\
\partial_{z_I}f_\theta(z)&=\left.\frac{d}{d \varepsilon}\right|_{\varepsilon=0}f(e^{\im \theta}(z+\im \varepsilon)=e^{\im \theta}\left.\frac{d}{d \varepsilon}\right|_{\varepsilon=0}f(z+\im \varepsilon)=e^{\im \theta}\partial_{z_I}f(z),\\
\end{align*}
Therefore, we have the conclusion.
\end{proof}

\begin{proof}[Proof of (iii) of Lemma \ref{lem:contcoo}]
We set $c_{JA}:=C_{JA}'(e^{\im \theta}z)(e^{\im \theta}\Xi)$ and $\partial_{J'R}\Phi [0]=\psi_0$, $\partial_{J'I}\Phi [0]=\im\psi_0$.
Then, since $C_{JA}'(e^{\im \theta}z)(e^{\im \theta}\Xi)$ is the unique solution of \eqref{eq:contcoo5} with $z,\Xi$ replaced by $e^{\im \theta}z,e^{\im \theta}\Xi$, we have
\begin{align*}
0&=\Omega\(\partial_{z_{JR}}\Phi[e^{\im \theta}z]|{B}^{-\frac{\sigma _3}{2}}e^{\im \theta}\Xi + \sum _{J'=1}^{2n} c_{J'R} \psi_0 + \im c_{J'I}\psi_0\)\\&=
\cos\theta\Omega\(\partial_{z_{JR}}\Phi[z]|{B}^{-\frac{\sigma _3}{2}}\Xi + \sum _{J'=1}^{2n} (\cos \theta c_{J'R}+\sin\theta c_{J'I})\psi_0+(\cos\theta c_{J'I}-\sin\theta c_{J'R})\im \psi_0\)\\&\quad
-\sin\theta\Omega\(\partial_{z_{JI}}\Phi[z]|{B}^{-\frac{\sigma _3}{2}}\Xi + \sum _{J'=1}^{2n} (\cos \theta c_{J'R}+\sin\theta c_{J'I})\psi_0+(\cos\theta c_{J'I}-\sin\theta c_{J'R})\im \psi_0\)
\end{align*}
and
\begin{align*}
0&=\Omega\(\partial_{z_{JI}}\Phi[e^{\im \theta}z]|{B}^{-\frac{\sigma _3}{2}}e^{\im \theta}\Xi + \sum _{J'=1}^{2n} c_{J'R} \psi_0 + \im c_{J'I}\psi_0\)\\&=
\sin\theta\Omega\(\partial_{z_{JR}}\Phi[z]|{B}^{-\frac{\sigma _3}{2}}\Xi + \sum _{J'=1}^{2n} (\cos \theta c_{J'R}+\sin\theta c_{J'I})\psi_0+(\cos\theta c_{J'I}-\sin\theta c_{J'R})\im \psi_0\)\\&\quad
+\cos\theta\Omega\(\partial_{z_{JI}}\Phi[z]|{B}^{-\frac{\sigma _3}{2}}\Xi + \sum _{J'=1}^{2n} (\cos \theta c_{J'R}+\sin\theta c_{J'I})\psi_0+(\cos\theta c_{J'I}-\sin\theta c_{J'R})\im \psi_0\),
\end{align*}
for all $J$.
Therefore, we have
\begin{align*}
0=\Omega\(\partial_{z_{JA}}\Phi[z]|{B}^{-\frac{\sigma _3}{2}}\Xi + \sum _{J'=1}^{2n} (\cos \theta c_{J'R}+\sin\theta c_{J'I})\psi_0+(\cos\theta c_{J'I}-\sin\theta c_{J'R})\im \psi_0\),
\end{align*}
for all $J,A$.

Since $C_{J'A'}'(z)\Xi$ is a unique solution of \eqref{eq:contcoo5}, we have
\begin{align*}
C_{J'R}'(z)\Xi=\cos \theta c_{J'R}+\sin \theta c_{J'I},\quad C_{J'I}'(z)\Xi = -\sin \theta c_{J'R}+\cos \theta C_{J'I}.
\end{align*}
Therefore, we obtain,
\begin{align*}
C_{J'R}'(e^{\im \theta}z)(e^{\im \theta}\Xi)&=\cos\theta C'_{J'R}(z)\Xi-\sin \theta C'_{J'I}(z)\Xi,\\
C_{J'R}'(e^{\im \theta}z)(e^{\im \theta}\Xi)&=\sin\theta C'_{J'R}(z)\Xi+\cos \theta C'_{J'I}(z)\Xi,
\end{align*}

Finally, notice that it suffices to show $R[e^{\im \theta}z](e^{\im \theta}\Xi) = e^{\im \theta } R[z]\Xi$.
\begin{align*}
&R[e^{\im \theta}z](e^{\im \theta}\Xi)=B^{-\frac{\sigma}{2}}e^{\im \theta}\Xi + \sum_{J'=1}^{2n}(C_{J'R}'(e^{\im \theta}z)e^{\im \theta}\Xi)\psi_0+(C_{J'I}'(e^{\im \theta}z)e^{\im \theta}\Xi)\im \psi_0=\\&
e^{\im \theta}B^{-\frac{\sigma}{2}}\Xi + \sum_{J'=1}^{2n}(C_{J'R}'(\cos\theta C'_{J'R}(z)\Xi-\sin \theta C'_{J'I}(z)\Xi)\psi_0+(\sin\theta C'_{J'R}(z)\Xi+\cos \theta C'_{J'I}(z)\Xi)e^{\im \theta}\Xi)\im \psi_0\\&
=e^{\im \theta}\(B^{-\frac{\sigma}{2}}\Xi + \sum_{J'=1}^{2n}(C_{J'R}'(z)\Xi)\psi_0+(C_{J'I}'(z)\Xi)\im \psi_0\)=e^{\im \theta}R[z]\Xi.
\end{align*}
Therefore, we have the conclusion.
\end{proof}

\section*{Acknowledgments}   S.C. was partially funded      grants FIRB 2012 (Dinamiche Dispersive) from the Italian Government  and   FRA 2013 from the University of Trieste.
M.M. was supported by the Japan Society for the Promotion of Science (JSPS) with the Grant-in-Aid for Young Scientists (B) 15K17568. T. P. is partly supported by Simons Foundation, grant \# 354889.

\bibliographystyle{amsplain}

\begin{thebibliography}{CP03}


\bibitem{AS}
  X.An, A.Soffer, {\em Remark on Asymptotic Completeness for Nonlinear Klein-Gordon Equations with Metastable States},  arXiv:1510.06485.




\bibitem {bambusicuccagna}
  D.Bambusi, S.Cuccagna, {\em On dispersion of
small energy solutions of the nonlinear Klein Gordon equation with a
potential},   Amer. Math. Jour.    {\bf   133} (2011), 1421--1468.


\bibitem{boussaidcuccagna}
  N.Boussaid, S.Cuccagna,  {\em  On stability of standing waves of nonlinear Dirac equations\/},    Comm. in Partial Diff. Eq.  37  (2012),   1001--1056.


\bibitem{BP1}
V.Buslaev, G.Perelman, {\em Scattering for the nonlinear
Schr\"odinger equation: states close to a soliton\/},  St.
Petersburg Math.J., 4
 (1993), pp.  1111--1142.



\bibitem{BP2}
V.Buslaev, G.Perelman, {\em On the stability of solitary waves for
nonlinear Schr\"odinger equations},   Nonlinear evolution
equations, editor N.N. Uraltseva, Transl. Ser. 2, 164, Amer. Math.
Soc.,
   pp.  75--98, Amer. Math. Soc., Providence (1995).



\bibitem{Cu0}
  S.Cuccagna,  {\em On the Darboux and Birkhoff   steps in the asymptotic
  stability of    solitons},  Rend. Istit. Mat. Univ. Trieste   {\bf  44} (2012), 197--257.
	
 			
\bibitem{Cu2}
  S.Cuccagna,  {\em The Hamiltonian structure of the nonlinear
Schr\"odinger equation and   the  asymptotic stability of its
ground states}, Comm. Math. Physics
  {\bf  305}  (2011),  279--331.
	

\bibitem{CM} S.Cuccagna, M.Maeda, {\em On weak interaction between  a ground state and a non--trapping  potential},  J. Differential Equations 256 (2014),   1395--1466.

\bibitem{CM1} S.Cuccagna, M.Maeda, {\em On   small energy stabilization    in the NLS with a trapping potential}, 	  Anal. PDE 8 (2015),   1289--1349.


\bibitem{CM2} S.Cuccagna, M.Maeda, {\em On orbital instability of spectrally stable vortices of the NLS in the plane}, 	   arXiv:1508.03146.


\bibitem{CT} S.Cuccagna, M.Tarulli, {\em On   stabilization   of small solutions   in the nonlinear Dirac equation with a trapping potential}, 	 arXiv:1309.4878.

\bibitem{danconafanelli}
P. D'Ancona, L. Fanelli, {\em Strichartz and smoothing estimates for
dispersive equations with magnetic potentials\/},    Comm. P.D.E., 33(6)
 (2008), pp.  1082-1112.


\bibitem {friedlander}
F.G. Friedlander, {\em The wave equation on a curved space-time}, Cambridge Un. Press  (1975).


\bibitem {zhousigal}
Zhou Gang, I.M.Sigal, {\em Relaxation of Solitons in Nonlinear
  Schr\"odinger Equations with Potential \/}, Advances in
  Math.,
  216 (2007), pp. 443-490.


\bibitem{GW1}
Zhou Gang, M.I.Weinstein, {\em Dynamics of Nonlinear
Schr\"odinger/Gross-Pitaeskii Equations; Mass transfer in Systems
with Solitons and Degenerate Neutral Modes},  Anal. PDE {\bf  1 }   (2008),
   267--322.

\bibitem {GW2}
Zhou Gang, M.I.Weinstein, {\em Equipartition of Energy in Nonlinear
Schr\"odinger/Gross-Pitaeskii Equations},   Appl. Math. Res. Express. AMRX    {\bf  2} (2011), 123--181.







 \bibitem{GSS}   M.Grillakis, J.Shatah, W.Strauss, {\em Stability
of solitary waves in the presence of symmetries, I },  Jour. Funct.
An.   {\bf 74} (1987),  pp.160--197.

\bibitem{GNT}
S.Gustafson,   K.Nakanishi, T.P.Tsai,  {\em Asymptotic stability and completeness in the energy space for nonlinear Schr\"odinger equations with small solitary waves}, Int. Math. Res. Not.  2004 (2004) no.  {\bf  66}, 3559--3584.


\bibitem{GP}
S.Gustafson,    T.V.Phan,  {\em Stable directions for degenerate excited states of nonlinear  Schr\"odinger equations}, SIAM J. Math. Anal. {\bf  43} (2011) , 1716--1758.






\bibitem{maeda} M.Maeda, {\em    Existence and asymptotic stability of quasi-periodic solution of discrete NLS with potential in $\mathbb{Z}$},  arXiv:1412.3213 .

\bibitem {N}K.Nakanishi,  {\em Global dynamics below excited solitons for the nonlinear Schr\"odinger equation with a potential},  arXiv:1504.06532.


\bibitem {NPT}K.Nakanishi, T.V.Phan, T.P.Tsai, {\em Small solutions of nonlinear Schr\"odinger equations near first excited states}, Jour. Funct. Analysis  {\bf  263}  (2012), 703--781.

\bibitem {PelinovskyStefanov}
D.Pelinovsky, A.Stefanov.
{\em Asymptotic stability of small gap solitons in the nonlinear dirac
  equations},   J. Math. Phys. 53 (2012),   073705, 27 pp.

\bibitem{PiW}
C.A.Pillet, C.E.Wayne, {\em Invariant manifolds for a class of
dispersive, Hamiltonian partial differential equations} J. Diff. Eq.
  141 (1997), pp. 310--326.


\bibitem {SW1} A.Soffer, M.I.Weinstein, {\em  Multichannel nonlinear
scattering for nonintegrable equations \/}, Comm. Math. Phys., 133
(1990), pp. 116--146






\bibitem {SW2}
A.Soffer, M.I.Weinstein, {\em  Multichannel nonlinear scattering II.
The case of anisotropic potentials and data \/},  J. Diff. Eq., 98
 (1992), pp.
  376--390.



\bibitem{SW4}    A.Soffer, M.I.Weinstein,
{\em Selection of the ground state for nonlinear Schr\"odinger
equations}, Rev. Math. Phys.      {\bf  16}  (2004),   977--1071.



\bibitem{SW3}
A.Soffer, M.I.Weinstein, {\em Resonances, radiation damping and
instability in Hamiltonian nonlinear wave equations},  Invent.
Math.   {\bf  136}
 (1999),
  9--74.








\bibitem {taylor}
M.E. Taylor, {\em Partial Differential Equations II}, App. Math. Sci. 116, Springer, New York (1997).

\bibitem {TY1}
  T.P.Tsai, H.T.Yau, {\em Asymptotic dynamics of nonlinear
Schr\"odinger equations: resonance dominated and radiation dominated
solutions}, Comm. Pure Appl. Math.  {\bf 55}  (2002),   153--216.

\bibitem {TY2}
  T.P.Tsai, H.T.Yau, {\em Relaxation of excited states in
nonlinear Schr\"odinger equations}, Int. Math. Res. Not.  {\bf 31}
(2002),   1629--1673.

\bibitem {TY3}
{  T.P.Tsai, H.T.Yau}, {\em Classification of asymptotic profiles
for nonlinear Schr\"odinger equations with small initial data}, Adv.
Theor. Math. Phys.  {\bf 6} (2002),   107--139.


\bibitem {TY4}
{  T.P.Tsai, H.T.Yau}, {\em Stable directions for
excited states of nonlinear Schr\"odinger equations},   Comm.
P.D.E.     {\bf 27} (2002),  2363--2402.



\bibitem{Y1}   K.Yajima, The $W^{k,p}$-continuity of wave operators
for Schr\"{o}dinger operators, J. Math. Soc. Japan,  {47} (1995),
pp. 551--581.

		

		
		












	\end{thebibliography}

Department of Mathematics and Geosciences,  University
of Trieste, via Valerio  12/1,  Trieste,  34127,  Italy. {\it E-mail Address}: {\tt scuccagna@units.it}
\\

Department of Mathematics and Informatics,
Faculty of Science,
Chiba University,
Chiba 263-8522, Japan.
{\it E-mail Address}: {\tt maeda@math.s.chiba-u.ac.jp}
\\

Department of Mathematics,
University of Tennessee, 227 Ayres Hall, 1403 Circle Drive,
Knoxville, TN 37996-1320, United States.
{\it E-mail Address}: {\tt phan@math.utk.edu}

\end{document}